\documentclass[a4paper,11pt]{amsart}
\usepackage{amssymb,amsfonts,amsmath,amsxtra,mathtools,mathrsfs,placeins,graphicx,verbatim,stmaryrd,hyperref,cite,color}
\usepackage[all]{xy}
\usepackage[T1]{fontenc}
\usepackage{lmodern}
\xyoption{line}
\usepackage{mparhack}
\usepackage{fullpage}
\usepackage{cite}
\newtheorem{theorem}{Theorem}[section]
\newtheorem{cor}[theorem]{Corollary}
\newtheorem{lem}[theorem]{Lemma}
\newtheorem{prop}[theorem]{Proposition}

\theoremstyle{definition}

\newtheorem{example}[theorem]{Example}
\newtheorem{defi}[theorem]{Definition}
\newtheorem{rem}[theorem]{Remark}

\newcommand{\yh}[1]{{\textcolor{blue}{yh: #1}}}
\newcommand{\tr}[1]{{\textcolor{blue}{Rong: #1}}}
\newcommand{\huaG}{\mathcal{G}}
\numberwithin{equation}{section}

\DeclareMathOperator{\CDGA}{CDGA}
\DeclareMathOperator{\im}{Im}

\DeclareMathOperator{\Hom}{Hom}

\DeclareMathOperator{\Der}{Der}
\DeclareMathOperator{\Derbar}{\overline{Der}}
\DeclareMathOperator{\DGVect}{DGVect}
\DeclareMathOperator{\ad}{ad}
\newcommand{\half}{\frac{1}{2}}

\def\ground{\mathbf{k}}
\def\RB{\operatorname{RB}}
\def\g{\mathfrak g}
\def\h{\mathfrak h}

\def\MC{\operatorname{MC}}
\def\VMC{\operatorname{VMC}}
\def\id{\operatorname{id}}
\newcommand{\Id}{{\rm{Id}}}

\newcommand{\frki}{\mathfrak i}

\newcommand{\dgla}{{\rm dgla}}

\newcommand{\lon}{\,\rightarrow\,}
\newcommand{\gl}{\mathfrak {gl}}
\newcommand{\huaF}{\mathcal{F}}

\newcommand{\huaX}{\mathcal{X}}

\newcommand {\emptycomment}[1]{}

\newcommand{\huaL}{\mathcal{L}}
\newcommand{\HRB}{\mathrm{HRB}}
\newcommand{\HLR}{L_\infty\mathrm{Rep}}
\newcommand{\HL}{\mathrm{HL}}

\newcommand{\LHRB}{\mathsf{HRB}}
\newcommand{\LHM}{\mathrm{T}\mathsf{L_{\infty}}\mathrm{Bi}}

\thanks{}

\begin{document}
	\bibliographystyle{hsiam2}

	\title{Homotopy relative Rota-Baxter Lie algebras, triangular $L_\infty$-bialgebras and higher derived brackets}
\author{Andrey Lazarev}
\address{Department of Mathematics and Statistics, Lancaster University, Lancaster LA1 4YF, UK}
\email{a.lazarev@lancaster.ac.uk}

\author{Yunhe Sheng}
\address{Department of Mathematics, Jilin University, Changchun 130012, Jilin, China}
\email{shengyh@jlu.edu.cn}

\author{Rong Tang}
\address{Department of Mathematics, Jilin University, Changchun 130012, Jilin, China}
\email{tangrong@jlu.edu.cn}
\begin{abstract}
We describe $L_\infty$-algebras governing homotopy relative Rota-Baxter Lie algebras  and triangular $L_\infty$-bialgebras, and establish
a map between them. Our formulas are based on a functorial approach to Voronov's higher derived brackets construction which is of independent interest. 	
\end{abstract}


	\maketitle
\tableofcontents
\section{Introduction}

The subject of this paper is the study of two important algebraic structures: homotopy relative Rota-Baxter (RB) Lie algebras and triangular $L_\infty$-bialgebras   and their relationship.
\subsection{Higher derived brackets} Given an inclusion of differential graded Lie algebras (dglas) $\frki:\g\to L$ with a direct complement $\h$ (so that $L\cong \h\oplus\g$), a homotopy fiber of $\frki$ is quasi-isomorphic to $\h[-1]$, the desuspension of $\h$. Under the additional assumption that the Lie bracket on $\g$ restricted to $\h$ vanishes, T. Voronov constructed  $L_\infty$-structures on the cocylinder and homotopy fiber of $\frki$ in \cite{Vor}. The higher products of these $L_\infty$-algebras are called \emph{higher derived brackets}. This construction, subsequently generalized in \cite{Bandiera, Bordemann} proved to be extremely useful and showed up in a variety of situations such as the study of simultaneous deformations of two compatible structures in \cite{Fregier-Zambon-1,FZ},
quantization of coisotropic submanifolds of Poisson manifolds \cite{Cattaneo-Felder} and many others. We present a functorial approach to higher derived brackets and identify explicitly Maurer-Cartan (MC) elements in the corresponding $L_\infty$-algebras. This is our first collection of results that are subsequently applied to the two concrete cases of interest (homotopy relative  Rota-Baxter Lie algebras and triangular $L_\infty$-bialgebras); it is clear that there are many other situations where they are relevant.

\subsection{Homotopy relative Rota-Baxter Lie algebras}
The concept of a Rota-Baxter (RB) operator was introduced by G. E. Baxter \cite{Baxter} with motivation from probability theory; it was later put in an abstract context as an operator on an associative algebra, satisfying a certain identity \cite{Rota}. Recently, it played an important role in the Connes-Kreimer's study of renormalization in quantum field theory \cite{CK}.  In the context of Lie algebras, RB operators are closely related with the classical Yang-Baxter equation \cite{Kuper} and thus, with the study of integrable systems, see the book \cite{Gub} for more details.

It is well-known that many homotopy invariant algebraic structures are themselves MC elements in certain dglas; this point of view underlies the modern approach to algebraic deformation theory in characteristic zero (cf. for example the survey \cite{GLTS} explaining this). In such a case we say that an algebraic structure is \emph{governed} by the corresponding dgla. In the previous work \cite{LTS}, the authors introduced the notion of a strong homotopy version of a RB Lie algebra, a so-called homotopy relative RB Lie algebra as an $L_\infty$-algebra together with an appropriate generalization of a RB operator. In the present paper we explicitly find an $L_\infty$-algebra governing this algebraic structure and, using our functorial approach to higher derived brackets, express the homotopy RB identities in a compact, `synthetic' way.

\subsection{Triangular $L_\infty$-bialgebras}
A triangular Lie bialgebra is a Lie bialgebra $\g$ whose cobracket $\g\to\Lambda^2\g$ is the coboundary of an $r$-matrix, i.e. a skew-symmetric quadratic element satisfying the classical Yang-Baxter equation $[r,r]=0$ in the Schouten Lie algebra $\Lambda^*\g$ of $\g$. Thus, a triangular Lie bialgebra can be viewed as a pair $(\g,r)$ where $\g$ is a Lie algebra and
$r\in\Lambda^2\g$ is an $r$-matrix. Triangular Lie bialgebras play an important role in deformation quantization \cite{CP,Merkulov,Resh}; in particular, it is known that they can be quantized \cite{Etingof-Kazhdan}, in the sense that their universal enveloping algebras admit formal deformations as \emph{triangular} Hopf algebras, an important notion that goes beyond the scope of this paper and will not be discussed here.

Just as a Lie algebra has a strong homotopy generalization, called an $L_\infty$-algebra, Lie bialgebras have a strong homotopy version called $L_\infty$-bialgebras, cf. \cite{BSZ,Kra}. More recently, a strong homotopy generalization of a triangular Lie bialgebra was introduced in \cite{BVor}, together with an appropriate infinity analogue of an $r$-matrix, the so-called $r_\infty$-matrix. A triangular $L_\infty$-bialgebra can then be defined as a pair $(\g,r)$ where $\g$ is an $L_\infty$-algebra and $r$ is a (suitably defined) $r_\infty$-matrix. One of the motivations behind studying triangular $L_\infty$-bialgebras is the problem of quantizing them and, in particular, making sense of the notion of a \emph{quantum} $r_\infty$-matrix (or, in other words, $A_\infty$ quantum Yang-Baxter equation).

It is  well-known that $L_\infty$-bialgebras with a fixed underlying graded vector space are  governed by a certain dgla. This point of view, for example, allows interpreting deformation quantization of Lie bialgebras as an $L_\infty$-quasi-isomorphism between dglas (or $L_\infty$-algebras) governing $L_\infty$-bialgebras and strong homotopy associative bialgebras, cf. \cite{Merkulov} regarding this approach. By contrast, triangular $L_\infty$-bialgebras are not governed by any dgla but rather, by a certain $L_\infty$-algebra that we explicitly identify. This suggests that the perspective of \cite{Merkulov} may be applicable to quantization of triangular $L_\infty$-bialgebras.

Finally, we show that a triangular $L_\infty$-bialgebra gives rise to an associated homotopy relative RB Lie algebra; this correspondence comes from a map between the $L_\infty$-algebras governing the corresponding structures. This substantially strengthens a result in the authors' previous paper \cite{LTS} where it was proved in the ungraded case and only on the level of cohomology. It is interesting to observe that, while the construction of homotopy RB Lie algebras carries over easily to the associative (or $A_\infty$) context, the corresponding analogue of triangular $L_\infty$-bialgebras (that we can putatively call  \emph{triangular $A_\infty$-Hopf algebras}) is rather less straightforward to construct. We hope to return to this problem in a future work.
\subsection{Notation and conventions}

We work in the category $\DGVect$ of differential graded (dg) vector spaces over a field $\ground$ of characteristic zero; the grading is always cohomological.    The \emph{n-fold suspension}  of a  graded vector space $\g$ is defined by the convention   $ \g[n]^i= \g^{i+n}$; the differential is therefore a map $d:\g\to \g[1]$. Given an element $x\in\g^k$, the corresponding element in $\g[n]^{k-n}$ will be denoted by $x[n]$.  There is an isomorphism $(\g[n])^*\cong \g^*[-n]$.

The category $\DGVect$ is symmetric monoidal, and monoids  (respectively  commutative  monoids)  in it are called dg algebras (dgas) and commutative dg algebras (cdgas). We also need to work with \emph{pseudocompact} dg vector spaces, or projective limits of finite-dimensional vector spaces; thus a pseudocompact dg vector space $V$ can be written as $V=\varprojlim_\alpha V_{\alpha}$ for a projective system $\{V_\alpha\}$ of finite dimensional dg vector spaces. The category of pseudocompact dg vector spaces is equivalent to the opposite category of $\DGVect$ with anti-equivalence established by the $\ground$-linear duality functor. This category also admits a symmetric monoidal structure which we will denote simply by $\otimes$
and monoids (respectively commutative monoids) in it are called pseudocompact dgas (respectively pseudocompact cdgas). We will call local pseudocompact dgas (or cdgas) \emph{complete}, another name for a complete (c)dga is a local pro-Artinian (c)dga. Occasionally we need to consider the tensor product of a pseudocompact dg vector space $V=\varprojlim_\alpha V_{\alpha}$ and a discrete one $U$; in this situation we will always write $V\otimes U$ for $\varprojlim_\alpha V_{\alpha}\otimes U$; such a tensor product is in general neither discrete nor pseudocompact. The category of complete cdgas and continuous multiplicative maps will be denoted by $\CDGA_\ground^\wedge$. The maximal ideal of a complete cdga $A$ will be denoted by $A_{\geq 1}$;
thus $A=\ground \oplus A_{\geq 1}$; there is a filtration
$A=:A_{\geq 0}\supset A_{\geq 1}\supset \ldots\supset A_{\geq n}\supset\ldots$
where $A_{\geq n}$ stands for the $n$-th power of the maximal ideal.

We will also need the notion of a dgla; this is a dg vector space $(\g,d)$ with an anti-symmetric product $[-,-]:\g\otimes\g\to \g$ satisfying the graded Jacobi identity. A Maurer-Cartan (MC) element in a dgla $\g$ is an element $x\in\g^1$ such that $dx+\frac{1}{2}[x,x]=0$; the set of MC elements in a dgla $\g$ will be denoted by $\MC(\g)$.

For a graded Lie algebra $\g$ and $x\in\g$ we denote by $\ad_x$ the \emph{right} adjoint action of $x$ defined by  $\ad_xy=[y,x]$ for $y\in \g$.

\section{$L_\infty$-algebras and their extensions}
Let $\g$ be a graded vector space. We denote by $\Der\hat{S}\g^*[-1]$  the graded Lie algebra consisting of continuous derivations of the complete symmetric algebra $\hat{S}\g^*[-1]$ and by $\Derbar\hat{S}\g^*[-1]$  its graded Lie subalgebra of derivations vanishing at zero. We will briefly recall the definition of an $L_\infty$-algebra following e.g. \cite{Laz, Laz'}. See also \cite{LS,LM} for more details and a different point of view.

\begin{defi}
An {\em $L_\infty$-algebra} structure  on $\g$ is a continuous degree $1$ derivation $m$ of the complete cdga $\hat{S}\g^*[-1]$,  such that $m\circ m=0$ and $m$ has no constant term. The pair $(\g,m)$ is called
an $L_\infty$-algebra, and $(\hat{S}\g^*[-1],m)$ is called its representing complete cdga. Sometimes we will refer to $\g$ as an $L_\infty$-algebra leaving $m$ understood.
\end{defi}

\begin{rem}\label{homotopy-lie-mc}
Thus, an $L_\infty$-algebra structure is an MC element in the graded Lie algebra $\Derbar\hat{S}\g^*[-1]$. If the full graded Lie algebra $\Der\hat{S}\g^*[-1]$ is taken in place of  $\Derbar\hat{S}\g^*[-1]$ we get the definition of a  \emph{curved} $L_\infty$-algebra. Many of our results hold, with appropriate modifications, for curved $L_\infty$-algebras.
\end{rem}

\begin{rem}\label{rmk:10}
Let $(\g,m)$ be an $L_\infty$-algebra. The element $m$ can be written as a sum
$$
m=m_1+\cdots+m_n+\cdots
$$
where $m_n$ is the  order $n$  part of $m$ so we can write $m_n: \g^*[-1]\to \hat{S}^n\g^*[-1]$. Consider the dual map of $m_n$, thus we have the degree $1$ map $\check{m}_n:S^n\g[1]\to  \g[1]$ for $n=1,2,\cdots.$ Namely an $L_\infty$-algebra on a graded vector space $\g$ is a sequence of linear maps of degree $1$:
 $$
 \check{m}_n:S^n\g[1]\to  \g[1],\quad n\ge1
 $$
which satisfy the following relations for any homogeneous elements $x_1,\cdots,x_n\in \g[1]$:
\begin{eqnarray}\label{homotopy-lie}
\sum_{i=1}^{n}\sum_{\sigma\in \mathbb S_{(i,n-i)} }\varepsilon(\sigma;x_1,\cdots,x_n)\check{m}_{n-i+1}(\check{m}_i(x_{\sigma(1)},\cdots,x_{\sigma(i)}),x_{\sigma(i+1)},\cdots,x_{\sigma(n)})=0.
\end{eqnarray}
\emptycomment{
In the sequel, $(\g[1],\{\check{m}_n\}_{n=1}^{\infty})$  will be referred as an $L_\infty[1]$-algebra.  That is,  an $L_\infty$-algebra s is equivalent to that $(\g[1],\{\check{m}_n\}_{n=1}^{\infty})$  is an $L_\infty[1]$-algebra.

It is an $L_\infty[1]$-algebra structure on $\g[1]$. \emptycomment{We define degree $2-n$ higher products $l_n$ in $\g$, which are given by
\begin{eqnarray*}
l_n(x_1,\cdots,x_n)=(-1)^{\sum_{i=1}^{n}(n-i)|x_i|}\check{m}_n(x_1[1],\cdots,x_n[1])[-1],\quad x_1,\cdots,x_n\in\g.
\end{eqnarray*}
By the definition of an $L_\infty$-algebra, we deduce that $l_n$ are graded antisymmetric and the following relations are satisfied for all $n\ge 1:$
\begin{eqnarray*}\label{sh-Lie}
\sum_{i=1}^n\sum_{\sigma\in \mathbb S_{(i,n-i)} }(-1)^{i(n-i)}(-1)^\sigma \varepsilon(\sigma;x_1,\cdots,x_n)l_{n-i+1}(l_i(x_{\sigma(1)},\cdots,x_{\sigma(i)}),x_{\sigma(i+1)},\cdots,x_{\sigma(n)})=0.
\end{eqnarray*}
Thus, we get the original definition of an $L_\infty$-algebra introduced by Schlessinger and   Stasheff  \cite{SS85,stasheff:shla},
Therefore,}It is well-know that $L_\infty[1]$-algebra structures on $\g[1]$ are equivalent to  $L_\infty$-algebra structures on $\g$. In this paper, both $L_\infty[1]$-algebras and $L_\infty$-algebras will be used according to different situations.}
\end{rem}
\emptycomment{
We will need the notion of  an MC set for an $L_\infty$-algebra, generalizing the corresponding notion for a dgla.}

\emptycomment{
\begin{defi}{\rm (\cite{Dolgushev-Rogers})}\label{filtration}
A {\em filtered $L_\infty$-algebra} is a pair $(\g,\huaF_{\bullet}\g)$, where $\g$ is an $L_\infty$-algebra and $\huaF_{\bullet}\g$ is a descending
filtration of the graded vector space $\g[1]$ such that $\g[1]=\huaF_1\g\supset\cdots\supset\huaF_n\g\supset\cdots$ and
\begin{itemize}
\item[\rm(i)]
for all $k\ge 1$ and $n_1,n_2,\cdots,n_k\ge 1$ we have
$$
\check{m}_k(\huaF_{n_1}\g,\huaF_{n_2}\g,\cdots,\huaF_{n_k}\g)\subset\huaF_{n_1+n_2+\cdots+n_k}\g,
$$
\item[\rm(ii)]
 $\g$ is complete with respect to the filtration, i.e.
$
\g[1]\cong\underleftarrow{\lim}~\g[1]/\huaF_n\g,
$
as $L_\infty$-algebras.
\end{itemize}
\end{defi}

\tr{Maybe we can delete the Definition \ref{filtration}. There are a lot of definitions in this Section.}
}

\emptycomment{
\begin{defi}
The set of {\em $\MC$ elements}, denoted by $\MC(\g)$, of a  filtered $L_\infty$-algebra \footnote{The definition of a  filtered $L_\infty$-algebra makes the MC equation well-defined.} $(\g,\huaF_{\bullet}\g)$ is the set of those $\alpha\in \g[1]^0$ satisfying the MC equation
\begin{eqnarray}\label{MC-equation}
\sum_{i=1}^{\infty}\frac{1}{i!}\check{m}_i(\alpha,\cdots,\alpha)=0.
\end{eqnarray}
\end{defi}
}

\begin{defi}\label{homotopy-lie-mor}
Let $(\g,m)$ and $(\h,m')$ be two $L_\infty$-algebras.	An {\em $L_\infty$-map} $f$ from $\g$ to $\h$ is, by definition, a continuous map of degree $0$ between the corresponding representing complete cdgas so that
$f:\hat{S}\h^*[-1]\to \hat{S}\g^*[-1]$.
\end{defi}
\begin{rem}
An $L_\infty$-map $f:\g\to\h$ can be represented as $f=f_1+\cdots+f_n+\cdots$ where $f_n$ is the order $n$ part of $f$ so that $f_n:\h^*[-1]\to \hat{S}^n\g^*[-1]$. If $f_n=0$ for $n\neq 1$ then $f$ is called a {\em strict $L_\infty$-map}. Furthermore, dualizing, an $L_\infty$-map $f:\g\to\h$ is a sequence of linear maps of degree $0$:
$$
\check{f_n}:S^n\g[1]\to \h[1],\quad n\ge 1
$$
such that the following relation holds for any homogeneous elements $x_1,\cdots,x_n\in \g[1]$:
\begin{eqnarray*}
&&\sum_{i=1}^{n}\sum_{\sigma\in \mathbb S_{(i,n-i)}}\varepsilon(\sigma;x_1,\cdots,x_n)\check{f}_{n-i+1}(\check{m}_i(x_{\sigma(1)},\cdots,x_{\sigma(i)}),x_{\sigma(i+1)},\cdots,x_{\sigma(n)})\\
&=&\sum_{k_1+\cdots+k_j=n}\sum_{\sigma\in \mathbb S_{(k_1,\cdots,k_j)}}\varepsilon(\sigma;x_1,\cdots,x_n)\\
&&\frac{1}{j!}\check{m}'_j(\check{f}_{k_1}(x_{\sigma(1)},\cdots,x_{\sigma(k_1)}),\cdots,\check{f}_{k_j}(x_{\sigma(k_1+\cdots+k_{j-1}+1)},\cdots,x_{\sigma(n)})).
\end{eqnarray*}
\emptycomment{
It is an $L_\infty[1]$-algebra homomorphism  from  $\g[1]$ to $\h[1]$.}
\end{rem}

\begin{defi}Let $(\g,m)$ be an $L_\infty$-algebra. Then an element $\xi\in(\g[1])^0$ is an {\em $\MC$ element}  if it satisfies the MC equation
	\begin{eqnarray}\label{eq:MC}
	\sum_{i=1}^{\infty}
	\frac{1}{i!}\check{m}_i(\xi,\cdots,\xi)=0.
	\end{eqnarray}
The set of MC elements in an $L_\infty$-algebra $(\g,m)$  will be denoted by $\MC(\g)$.
\end{defi}

\begin{rem}
The definition of an MC element in an $L_\infty$-algebra assumes that the left hand side of (\ref{eq:MC}) converges. If this is the case for all $\xi\in\g^1$, we will say that $\g$ contains MC elements. For example, if $m_n=0$ for $n>2$ (e.g. $\g$ is essentially a dgla) then $\g$ contains MC elements.	We now formulate some other sufficient conditions for an $L_\infty$-algebra to contain MC elements.
\end{rem}
\begin{defi}\label{def:filtered}
		Let $\g$ be an $L_\infty$-algebra and $\mathcal{F}_{\bullet}\g$ be a descending
		filtration of the graded vector space $\g$ such that $\g=\mathcal{F}_1\g\supset\cdots\supset\huaF_n\g\supset\cdots$ and
			$\g$ is complete with respect to this filtration, i.e. there is an isomorphism of graded vector spaces
			$
			\g\cong\underleftarrow{\lim}~\g/\mathcal{F}_n\g
			$.		
			\begin{enumerate}
				\item If for all $k,n_1,\ldots, n_k\geq 1$ it holds that
				\[
		\check{m}_k(\mathcal{F}_{n_1}\g[1],\cdots,\mathcal{F}_{n_k}\g[1])\subseteq \mathcal{F}_{n_1+\cdots+n_k}\g[1],		
				\]
	we say that the pair $(\g,\mathcal{F}_\bullet\g)$ is a \emph{filtered} $L_\infty$-algebra.	
	\item If there exists $l\geq 0$ such that for all $k>l$ it holds that
	\[
	\check{m}_k(\g[1],\cdots,\g[1])\subseteq \mathcal{F}_k\g[1],
	\]	we say that the pair $(\g,\mathcal{F}_\bullet\g)$ is a \emph{weakly filtered} $L_\infty$-algebra.		
			\end{enumerate}

\end{defi}
\begin{rem}
	The definition of a filtered $L_\infty$-algebra belongs to Dolgushev and Rogers \cite{Dolgushev-Rogers}. It is the \emph{weak} notion that is relevant to our immediate purposes while Dolgushev-Rogers's notion is given for comparison. Taking in the definition of a filtered $L_\infty$-algebra $l=0$ and $n_1=n_2=\cdots=n_k=1$, we obtain the condition of being weakly filtered, i.e. one notion is indeed stronger than the other. Furthermore, when $\g$ is finite-dimensional, it is easy to see that the condition of being weakly filtered is equivalent to requiring that the differential $m$ on the representing complete cdga $\hat{S}\g^*[-1]$ of $\g$ restricts to the (uncompleted) symmetric algebra $S\g^*[-1]$ whereas the stronger Dolgushev-Rogers condition means that the cdga $(S\g^*[-1],m)$ is cofibrant in the model category of cdgas.
	
	Furthermore, it is clear that a weakly filtered $L_\infty$-algebra contains MC elements.
\end{rem}
Given a complete cdga $A$ and an $L_\infty$-algebra $(\g,m)$, consider the tensor product $A_{\geq 1}\otimes \g$ and extend the $L_\infty$-structure maps $\check{m}_n:\g[1]^{\otimes n}\to  \g[1]$ by $A$-linearity to maps $$\check{m}_n^A:(A_{\geq 1}\otimes \g[1])^{\otimes n}\to A_{\geq 1}\otimes \g[1]$$ so that
\begin{equation*}
\check{m}_n^A(a_1\otimes x_1,\cdots, a_n\otimes x_n)=\begin{cases}
d_A(a_1)\otimes x_1+(-1)^{|a_1|}a_1\otimes \check{m}_1(x_1),&n=1,\\
(-1)^{\sum_{i=1}^{n}|a_i|(|x_1|+\cdots+|x_{i-1}|+1)} (a_1\cdots a_n)\otimes\check{m}_n( x_1,\cdots,x_n),&n\geq2,
\end{cases}
\end{equation*}
where $a_1,\cdots,a_n\in A_{\geq 1}$ and $x_1,\cdots,x_n\in\g[1]$. With this, $A_{\geq 1}\otimes \g$ becomes an $L_\infty$-algebra.
\begin{prop}
	Given an $L_\infty$-algebra $\g$ and a complete cdga $A$, the $L_\infty$-algebra $A_{\geq 1}\otimes \g$ is filtered.
\end{prop}
\begin{proof}
	We define a filtration on $A_{\geq 1}\otimes \g$ as follows:
	\[
	\mathcal{F}_n(A_{\geq 1}\otimes \g)=A_{\geq n}\otimes \g.
	\]
The corresponding conditions of Definition \ref{def:filtered} are trivial to check.
\end{proof}

Given a complete cdga $A$ and an $L_\infty$-map $f:\g\to\h$ with components $\check{f}_n$, there is
an $L_\infty$-map $f^A: A_{\geq 1}\otimes\g\to A_{\geq 1}\otimes \h$ with components $\check{f}^A_n$ defined by the formulas:
\begin{equation*}
\check{f}_n^A(a_1\otimes x_1,\cdots, a_n\otimes x_n)=
(-1)^{\sum_{i=1}^{n}|a_i|(|x_1|+\cdots+|x_{i-1}|)} (a_1\cdots a_n)\otimes\check{f}_n( x_1,\cdots,x_n),
\end{equation*}
where $a_1,\cdots,a_n\in A_{\geq 1}$ and $x_1,\cdots,x_n\in\g[1]$.
	
\begin{prop}
Let $(\g,m)$ be an $L_\infty$-algebra and $f:A\lon B$ be a morphism of complete cdgas. Then $f\otimes\Id_\g:A_{\geq 1}\otimes \g\lon B_{\geq 1}\otimes \g$ is a strict $L_\infty$-map.
\end{prop}
\begin{proof}
Since $f$ is a morphism of complete cdgas, it follows that $f(A_{\geq 1})\subset B_{\geq 1}$. For any $a_1\otimes x_1,\cdots, a_n\otimes x_n\in A_{\geq 1}\otimes \g[1]$, we have
\begin{eqnarray*}
&&(f\otimes\Id_\g)(\check{m}_n^A(a_1\otimes x_1,\cdots, a_n\otimes x_n))\\
&=&\begin{cases}
f(d_A(a_1))\otimes x_1+(-1)^{|a_1|}f(a_1)\otimes \check{m}_1(x_1),&n=1,\\
(-1)^{\sum_{i=1}^{n}|a_i|(|x_1|+\cdots+|x_{i-1}|+1)} f
(a_1\cdots a_n)\otimes\check{m}_n( x_1,\cdots,x_n),&n\geq2,
\end{cases}\\
&=&\check{m}_n^B\big((f\otimes\Id_\g)(a_1\otimes x_1),\cdots, (f\otimes\Id_\g)(a_n\otimes x_n)\big).
\end{eqnarray*}
Thus, we obtain that $f\otimes\Id_\g$ is a strict $L_\infty$-map.
\end{proof}

Moreover, given an  $L_\infty$-algebra $(\g,m)$, there is a set-valued functor $\MC_\g$ on the category $\CDGA_\ground^\wedge$, which is defined on the set of
objects and on the set of morphisms respectively by:
\begin{eqnarray}
\MC_\g(A)&=&\MC(\g,A),\\
\MC_\g(A\stackrel{f}{\lon}B)&=&\MC(\g,A)\stackrel{f\otimes\Id_\g}{\lon}\MC(\g,B),
\end{eqnarray}
for $A,B\in \CDGA_\ground^\wedge$ and $f\in \Hom_{\CDGA_\ground^\wedge}(A,B)$. Since $f\otimes\Id_\g$ is a strict $L_\infty$-map, we conclude that $f\otimes\Id_\g$ takes  elements in $\MC(\g,A)$ to elements in $ \MC(\g,B)$. So   the functor $\MC_\g$ is well-defined.
Moreover, it is representable.

\begin{theorem}\label{representable-thm}
Let $(\g,m)$ be an $L_\infty$-algebra. Then the functor $\MC_\g$ is represented by the complete cdga $(\hat{S}\g^*[-1],m)$. In other words, for any complete cdga $A$
there is an isomorphism \[\MC_{\g}(A)\cong\Hom_{\CDGA^{\wedge}}\left((\hat{S}\g^*[-1],m),A\right),\] functorial in $A$.
\end{theorem}
\begin{proof}
	See, e.g. \cite[Proposition 2.2 (1)]{CL'} where this result is proved in the $\mathbb{Z}/2$-graded setting but the arguments are the same  in the $\mathbb{Z}$-graded case.
\end{proof}	
 By the Yoneda embedding theorem, the functor $\MC_\g$ determines the $L_\infty$-algebra $\g$ up to a canonical $L_\infty$-isomorphism. Conversely, given a functor $\huaF$ on $\CDGA_\ground^\wedge$, we will often be interested in whether it is isomorphic to $\MC_\g$ for a suitable $L_\infty$-algebra $(\g,m)$.

Furthermore, for a fixed complete cdga $A$, the correspondence $\g\mapsto\MC(\g,A)$ is functorial with respect to
$L_\infty$-maps. Namely, the following result holds.
\begin{prop}\label{prop:MCmap}
	Let  $f:\g\to \h$ with components $(\check{f_1},\cdots, \check{f_n},\cdots)$. Then for any complete cdga $A$ it induces a map $f_*:\MC(\g,A)\to\MC(\h,A)$ according to the formula
\begin{equation}\label{eq:MCmap}
\xi\mapsto f_*(\xi)=\sum_{k=1}^{\infty} \frac{1}{k!}\check{f}_k^A(\xi,\cdots,\xi),\quad \forall \xi\in \MC(\g,A).
\end{equation}
\end{prop}
\begin{proof}
	 Since $f:\hat{S}\h^*[-1]\to \hat{S}\g^*[-1]$ is a map of complete cdgas, it clearly induces a map of sets
	\[ \MC(\g,A)\cong\Hom_{\CDGA^{\wedge}}(\hat{S}\g^*[-1],A)\to \Hom_{\CDGA^{\wedge}}(\hat{S}\h^*[-1],A)\cong \MC(\h,A),
	\]
	and a straightforward inspection shows that it is given in components by the stated formula.
\end{proof}	
\begin{rem}
	Proposition \ref{prop:MCmap} is well-known in formal deformation theory. It was formulated explicitly in \cite[Section 4.2]{Kon} for dglas and in  \cite[Section 2.5.5]{Mer2} in general.
\end{rem}
\begin{rem}
	All told, the set $\MC(-,-)$ can be viewed as a functor of two arguments. It is natural with respect to $L_\infty$-maps in the first argument and maps of complete cdgas in the second argument. In order to determine an $L_\infty$-map $\g\to\h$ it suffices to specify, for any complete cdga $A$,  a map $\MC(\g,A)\to\MC(\h,A)$, functorial in $A$ (by Yoneda's lemma). Moreover, it is clear that if $f=(\check{f_1},\cdots, \check{f_n},\cdots)$ is such that for any $\xi\in\MC(\g,A)$ the element $f_*(\xi)\in\MC(\h,A)$ given by formula \eqref{eq:MCmap} is an MC element, then $f$ is an $L_\infty$-map. This can sometimes be used for explicit constructions of $L_\infty$-maps out of MC elements.
\end{rem}
Recall that given a Lie algebra $\g$, its \emph{representation} in a vector space $V$ is  a Lie algebra map from $\g$ to $\gl(V)$. This generalizes in a straightforward way to the $L_\infty$-case.
\begin{defi}
Let $(\g,m)$ be an $L_\infty$-algebra with the representing cdga $(\hat{S}\g^*[-1],m)$ and $V$ be a graded vector space. Then a {\em representation} of $\g$ in $V$ is an $L_\infty$-map $f$ from $\g$ to $\gl(V)$, where $\gl(V)$ is the graded Lie algebra of endomorphisms of  $V$. 
\end{defi}

\begin{rem}\label{homotopy-rep-mc}
Let $\rho$ be a representation of an $L_\infty$-algebra $(\g,m)$ in a  graded vector space $V$. By Definition \ref{homotopy-lie-mor}, we deduce that $\rho\in\Hom_{\CDGA_\ground^\wedge}(\hat{S}\gl(V)^*[-1],\hat{S}\g^*[-1])$. Then, by Theorem \ref{representable-thm}, ${\rho}$ can be viewed as an MC element of the $\dgla$ $\hat{S}_{\geq 1}\g^*[-1]\otimes\gl(V)$.
\end{rem}

Given a complete cdga $A$, we will need the notion of an $A$-linear $L_\infty$-algebra; this notion, with a slight modification, was used in \cite{CL'}.

\begin{defi}\label{A-linear-homotopy-lie}
Let $A$ be a complete cdga and $\g$ be a graded vector space. Then an {\em $A$-linear $L_\infty$-algebra} structure on $A\otimes \g$ is an MC element of  the $\dgla$ $A_{\geq 1}\otimes\Derbar\hat{S}\g^*[-1]$. If $B$ is another complete cdga and $A\to B$ is a map, then 	an $A$-linear $L_\infty$-algebra on $A\otimes \g$ obviously determines a $B$-linear $L_\infty$-algebra on $B\otimes \g$ that will be referred to as obtained from $A\otimes \g$ by change of scalars.
\end{defi}

\begin{rem}
Note that an $A$-linear $L_\infty$-algebra structure on $A\otimes \g$ is a deformation of the trivial $L_\infty$-algebra structure on $\g$ with a dg base $A$. Alternatively, we could have called an $A$-linear $L_\infty$-algebra structure on $A\otimes \g$ an MC element of  $A\otimes\Derbar\hat{S}\g^*[-1]$. Our notion, slightly more restrictive, means that it is \emph{not} a generalization of an ordinary $L_\infty$-algebra over $\ground$ (because $\ground_{\geq 1}=0$).
\end{rem}

We now have the following result.
\begin{prop}\label{A-linear-homotopy-lie-rep}
	Let $\g$ be a graded vector space and $\huaF_{\HL}$ be the functor associating to a complete cdga $A$ the set of $A$-linear $L_\infty$-algebra structures on $A\otimes \g$. Then $\huaF_{\HL}$ is represented by the complete cdga $\hat{S}\big(\Derbar\hat{S}\g^*[-1]\big)^*[-1]\big)$.
\end{prop}
\begin{proof}
By Definition \ref{A-linear-homotopy-lie}, we deduce that $\huaF_{\HL}(A)=\MC(\Derbar\hat{S}\g^*[-1],A)$. Moreover, by Theorem \ref{representable-thm}, it follows that $\huaF_{\HL}(A)\cong\Hom_{\CDGA_\ground^\wedge}(\hat{S}\big(\big(\Derbar\hat{S}\g^*[-1]\big)^*[-1]\big),A)$ and we are done.
\end{proof}

\begin{cor}
Given a graded vector space $\g$, there exists a `universal' $\hat{S}\big(\big(\Derbar\hat{S}\g^*[-1]\big)^*[-1]\big)$-linear $L_\infty$-algebra structure on $\hat{S}\big(\big(\Derbar\hat{S}\g^*[-1]\big)^*[-1]\big)\otimes \g$ such that any other $A$-linear $L_\infty$-structure on $A\otimes \g$ is obtained by change of scalars  from a unique map of complete cdgas $\hat{S}\big(\big(\Derbar\hat{S}\g^*[-1]\big)^*[-1]\big)\to A$.
\end{cor}
\begin{proof}
By Proposition \ref{A-linear-homotopy-lie-rep}, there is a natural isomorphism $\alpha$ as following:
\begin{eqnarray*}
\alpha:\Hom_{\CDGA_\ground^\wedge}(\hat{S}\big(\big(\Derbar\hat{S}\g^*[-1]\big)^*[-1]\big),\cdot)\lon \huaF_{\HL}.
\end{eqnarray*}
By the universality of the representing cdga, we deduce that for any complete cdga $A$ and any $y\in\huaF_{\HL}(A)$, there exists a unique map of complete cdgas $f:\hat{S}\big(\big(\Derbar\hat{S}\g^*[-1]\big)^*[-1]\big)\to A$ such that $y=\huaF_{\HL}(f)(\huaX)$. Here $\huaX\in\huaF_{\HL}(\hat{S}\big(\big(\Derbar\hat{S}\g^*[-1]\big)^*[-1]\big))$ is given by
\begin{eqnarray*}
\huaX=\alpha\big({\hat{S}\big(\big(\Derbar\hat{S}\g^*[-1]\big)^*[-1]\big)}\big)(\Id).
\end{eqnarray*}
This completes the proof.
\end{proof}

\begin{rem}
	The universal $L_\infty$-algebra $\hat{S}\big(\big(\Derbar\hat{S}\g^*[-1]\big)^*[-1]\big)\otimes \g$ can be viewed as a universal deformation of the trivial $L_\infty$-algebra on $\g$; its universal properties persist upon passing to the homotopy category of complete cdgas cf. \cite{Laz', GLTS} regarding this approach to deformation theory. In the present paper we will not be concerned with this aspect of the theory.
\end{rem}

Suppose that we have an $A$-linear $L_\infty$-algebra structure $m_{A}$ on $A\otimes \g,$ where $A=(\hat{S}U^*[-1],m_U)$ is itself the representing complete cdga of an $L_\infty$-algebra $U$. The element $m_{A}$ is an $A$-linear derivation of
$$A\otimes \hat{S}\g^*[-1]=\hat{S}U^*[-1]\otimes \hat{S}\g^*[-1]\cong \hat{S}(U\oplus \g)^*[-1].$$
Forgetting that $m_{A}$ is $A$-linear, we can view it as an MC element in $\Derbar\hat{S}(U\oplus \g)^*[-1]$, i.e. an $L_\infty$-structure on $U\oplus \g$ with the representing cdga $(\hat{S}(U\oplus \g)^*[-1],m_{A})$. Moreover,  $A=(\hat{S}U^*[-1], m_U)$ is a sub-cdga of $(\hat{S}(U\oplus \g)^*[-1],m_{A})$ and we can form an $L_\infty$-structure on $\g$ with the representing cdga $\hat{S}\g^*[-1]$. All told, we have a sequence of $L_\infty$-algebras and strict $L_\infty$-maps:
\begin{equation}\label{eq:Lextension}
U\to U\oplus \g\to \g.
\end{equation}
This leads naturally to the notion of an \emph{extension} of $L_\infty$-algebras, cf. \cite{CL, Laz',Me}.
\begin{defi}
	The sequence of $L_\infty$-algebras and strict $L_\infty$-maps of the form \eqref{eq:Lextension} is called an {\em extension} of $\g$ by $U$.
\end{defi}
\begin{example}The universal $\hat{S}\big(\big(\Derbar\hat{S}\g^*[-1]\big)^*[-1]\big)$-linear $L_\infty$-algebra $$\hat{S}\big(\big(\Derbar\hat{S}\g^*[-1]\big)^*[-1]\big)\otimes \g$$ gives rise to an $L_\infty$-extension
	\[
	\g\to\Derbar\hat{S}\g^*[-1]\oplus \g\to\Derbar\hat{S}\g^*[-1].
	\]
Here $\g$ is given the trivial $L_\infty$-structure and the $L_\infty$-structure on $\Derbar\hat{S}\g^*[-1]\oplus \g$ can be read off the $\hat{S}\big(\big(\Derbar\hat{S}\g^*[-1]\big)^*[-1]\big)$-linear $L_\infty$-structure on $\hat{S}\big(\big(\Derbar\hat{S}\g^*[-1]\big)^*[-1]\big)\otimes \g$. Specifically (cf. \cite[Example 3.8]{CL}), for $\phi[1]\in\Hom(\g^{\otimes n}, \g)\subset\Der\hat{S}\g^*[-1]$ and $v_1,\cdots,v_n\in \g$ we have $$\check{m}_n(\phi[1],v_1,\cdots,v_n)=\phi(v_1,\cdots,v_n).$$

Note that the $L_\infty$-algebra $\Derbar\hat{S}\g^*[-1]\oplus \g$ represents the functor associating to a complete cdga $A$ an $A$-linear $L_\infty$-algebra on $A\otimes \g$ together with an MC element in it. Later on, we will consider a higher version of this construction with an MC element replaced with the so-called \emph{$r_\infty$-matrix}, cf. Definition \ref{def:r-infty} below associated to an $L_\infty$-algebra and see how that leads to triangular $L_\infty$-bialgebras.
\end{example}
Our next task is to describe a $\dgla$ controlling the pair of an $L_\infty$-algebra and  its representation in a  graded vector space. We arrange the set of such pairs as a functor on complete cdgas.

\begin{defi}
  An {\em $\HLR$ pair} consists of an $L_\infty$-algebra  $(\g,m)$  and a representation $\rho:\g\longrightarrow\gl(V)$   of $\g$ in a   graded vector space $V$.
\end{defi}

There is a natural  action  of the graded Lie algebra $\Derbar\hat{S}\g^*[-1]$ on  $\hat{S}_{\geq 1}\g^*[-1]\otimes\gl(V)$  given by
\begin{eqnarray*}
[\phi,x\otimes y]=\phi(x)\otimes y,\quad \forall \phi\in\Derbar\hat{S}\g^*[-1],~x\otimes y\in \hat{S}_{\geq 1}\g^*[-1]\otimes\gl(V).
\end{eqnarray*}
Let $\huaL_{\HLR}(\g,V)=\Derbar\hat{S}\g^*[-1]\ltimes \big(\hat{S}_{\geq 1}\g^*[-1]\otimes\gl(V)\big)$ be the corresponding semidirect product graded Lie algebra. Note that an $\HLR$ pair $((\g,m),\rho)$ is nothing but an MC element in the graded Lie algebra $\huaL_{\HLR}(\g,V)$.

\begin{defi}
Let $A$ be a complete cdga. Then an {\em $A$-linear $\HLR$ pair} with the underlying graded vector spaces $\g$ and $V$  is  an element in  $\MC(\huaL_{\HLR}(\g,V),A)$.
\end{defi}

Let  $\huaF_{\HLR}$ be the functor associating to a complete cdga $A$ the set of  $A$-linear $\HLR$ pairs with the  underlying graded vector spaces $\g$ and $V$. Then we have the following result.
\begin{prop}\label{homotopy-lie-and-rep-mc}
The functor $\huaF_{\HLR}$ is represented by the complete cdga $\hat{S}\huaL^*_{\HLR}(\g,V)[-1]$.
\end{prop}
\begin{proof}
Let $(m,x\otimes y)$ be a degree $1$ element in $\huaL_{\HLR}(\g,V)$. We have
\begin{eqnarray*}
[(m,x\otimes y),(m,x\otimes y)]=([m,m],2m(x)\otimes y+x^2\otimes [y,y]).
\end{eqnarray*}
Thus, $(m,x\otimes y)$ is an MC element of $\huaL_{\HLR}(\g,V)$ if and only if
\begin{eqnarray*}
[m,m]=0,\quad m(x)\otimes y+\frac{1}{2}x^2\otimes [y,y]=0.
\end{eqnarray*}
 By Remark \ref{homotopy-lie-mc} and Remark \ref{homotopy-rep-mc}, we deduce that $(m,x\otimes y)$ is an MC element of $\huaL_{\HLR}(\g,V)$ if and only if $(\g,m)$ is an $L_\infty$-algebra and $x\otimes y$ is a representation of the $L_\infty$-algebra $(\g,m)$ in a graded vector space $V$. Thus, we obtain that $\huaF_{\HLR}(A)=\MC(\huaL_{\HLR}(\g,V),A)$. Moreover, by Theorem \ref{representable-thm}, $\huaF_{\HLR}$ is represented by the complete cdga $\hat{S}\huaL^*_{\HLR}(\g,V)[-1]$.
\end{proof}	
\begin{rem}\label{lie-pair-gla}
	There is an inclusion of graded Lie algebras $i:\huaL_{\HLR}(\g,V)\subset\Derbar\hat{S}(\g\oplus V)^*[-1]$ where $\Derbar\hat{S}\g^*[-1]\subset \Derbar\hat{S}(\g\oplus V)^*[-1]$ in an obvious way and $\hat{S}_{\geq 1}\g^*[-1]\otimes\gl(V)\subset \Derbar\hat{S}(\g\oplus V)^*[-1]$ via the isomorphism $\hat{S}^n\g^*[-1]\otimes \gl(V)\cong \Hom(V^*[-1],\hat{S}^n\g^*[-1]\otimes V^*[-1])$; a simple check shows that the Lie bracket is preserved under this inclusion. By the proof of Proposition \ref{homotopy-lie-and-rep-mc}, the structure of an $L_\infty$-algebra on $\g$ together with a representation of $\g$ in a  graded vector space $V$ is equivalent to an MC element $(m,x\otimes y)\in \huaL_{\HLR}(\g,V)$ and thus, $i(m,x\otimes y)\in\MC \big(\Derbar\hat{S}(\g\oplus V)^*[-1]\big)$. The graded Lie algebra $\Derbar\hat{S}(\g\oplus V)^*[-1]$ supplied with the differential $d=[i(m,x\otimes y),\cdot]$ can be identified with the Chevalley-Eilenberg complex of the $L_\infty$-algebra $(\g,m)$ with coefficients in the representation $V$.
\end{rem}

\section{Voronov's higher derived brackets and MC elements}
In this section we review Voronov's constructions \cite{Vor} of higher derived brackets from the point of view of MC elements and $L_\infty$-extensions. Related results are contained in \cite{Fregier-Zambon-1}.

\begin{defi}\label{def:gauge}
Let $L$ be a dgla, $x\in\MC(L)$ and $h\in L^0$. The {\em right gauge transformation} by $h$ on $x$ is given by the following formula:
\[x\mapsto x*h:=x+\sum_{n=1}^{\infty}\frac{1}{n!}(\ad_h^n(x)+\ad_h^{n-1}(d(h))).\]
\end{defi}

\begin{rem} In the above definition it is assumed that the $e^{\ad_h}:=\sum_{n=0}^{\infty}\frac{(\ad_h)^n}{n!}$ is a well-defined operator on $L$.
	This is the case, e.g. when $L$ is a pronilpotent dgla.
	In that case it is easy to see that this is a well-defined action (i.e. $x*h\in \MC(L)$); moreover $x*h$ agrees with the ordinary (left) gauge action on $x$ by the element $-h$.
	\end{rem}

\begin{defi}
	We say that a dgla $L$ is supplied with a {\em $V$-structure} if there is given an operator $P:L\to L$ with $P^2=P$ (so that $P$ is a projector) such that \begin{enumerate}
		\item The subspace $\ker P$ is a sub-dgla of $L$,
		\item The  image of $P$ (denoted hereafter by $\h$) is an abelian graded Lie subalgebra of $L$.
	\end{enumerate}
From now on, we will denote a $V$-structure by a pair $(L,P)$.
If, for a given  $V$-structure, there is the following filtration on $L$:
\[
L\supset P[L,\h]\supset P[[L,\h],\h]\supset\cdots
\]
which is complete (e.g. if the adjoint action of $\h$ on $L$ is pronilpotent), the corresponding $V$-structure is called \emph{admissible}.
\end{defi}

\begin{rem}Note that for an admissible $V$-structure on $L$, the operator $e^{\ad_h}:L\to L$
makes sense for any $h\in\h^0$.
\end{rem}

Associated to an admissible $V$-structure on a  dgla $L$ is the notion of a VMC functor.

\begin{defi}
Let $(L,P)$ be an admissible $V$-structure. A {\em $\VMC$ element} associated to it is a pair $(x,h)$ where $x\in \MC(L)$ and $h\in \h^0$ such that $P(x*h)=0$. The set of VMC elements associated to $(L,P)$ will be denoted by $\VMC(L)$ (leaving $P$ understood).
\end{defi}
Just as the ordinary MC set in a dgla, the VMC set can be made into a functor of two arguments. Let $L$ be a dgla with a $V$-structure and $A$ be a complete cdga. Then $A_{\geq 1}\otimes L$ is pronilpotent and has an induced admissible $V$-structure given by the projector $\id\otimes P$.
\begin{defi}
	The {\em $\VMC$ set} of $L$ with values in $A$ is defined as $\VMC(A_{\geq 1}\otimes L)$ and denoted by $\VMC(L,A)$. 	
\end{defi}
It is clear that $\VMC(-,-)$ is a functor in the second variable. We will show that it is represented by a complete cdga that is the representing cdga of a certain $L_\infty$-algebra. Recall the following result by Voronov \cite{Vor}.
\begin{theorem}\label{thm:db-big-homotopy-lie-algebra}
	Let $(L,P)$ be a $V$-structure. Then the graded vector space $  L\oplus\h[-1]$  is an $L_\infty$-algebra where
	\begin{eqnarray*}\label{V-shla-big-algebra}
		\check{m}_1(x[1],h)&=&(-d(x)[1],P(x+d(h))),\\
		\check{m}_2(x[1],y[1])&=&(-1)^{|x|}[x,y][1],\\
		\check{m}_k(x[1],h_1,h_2,\cdots,h_{k-1})&=&P[\cdots[[x,h_1],h_2]\cdots,h_{k-1}],\quad k\geq 2,\\
		\check{m}_k(h_1,h_{2},\cdots,h_k)&=&P[\cdots[d(h_1),h_2]\cdots,h_{k}],\quad k\geq 2.
	\end{eqnarray*}
	Here $h,h_1,\cdots,h_k$ are homogeneous elements of $\h$ and $x,y$ are homogeneous elements of $L$. All the other $L_\infty$-algebra products that are not obtained from the ones written above by permutations of arguments, will vanish. \qed
\end{theorem}
\begin{rem}\label{rmk:extension}
	The $L_\infty$-products on $ L\oplus\h[-1]$ restrict to $\h[-1]$ making the latter into an $L_\infty$-algebra given by
\begin{eqnarray}\label{V-shla}
\check{m}_k(h_1,\cdots,h_k)=P[\cdots[d(h_1),h_2]\cdots,h_{k}],\quad\mbox{for homogeneous}~   h_1,\cdots,h_k\in\h.
\end{eqnarray}
It is included into an $L_\infty$-extension
	\begin{equation}\label{eq:extensionvoronov}
	\h[-1]\to  \h[-1]\oplus L\to L,
	\end{equation}
	where the second arrow is the natural projection and $L$ is viewed as a dgla (hence an $L_\infty$-algebra).
\end{rem}
For later use we record the following obvious observation.
\begin{rem}\label{thm:db-big-homotopy-lie-algebra-small}
	Let $L'$ be a graded Lie subalgebra of $L$ that satisfies $d(L')\subset L'$. Then $L'\oplus \h[-1]$ is an $L_\infty$-subalgebra of the above $L_\infty$-algebra $(L\oplus\h[-1],\{\check{m}_k\}_{k=1}^{\infty})$. 
\end{rem}
The following key lemma interprets a VMC element of an admissible V-structure as an MC element.
\begin{lem}\label{lem:VMC}
Let $L$ be a  dgla with an admissible V-structure. Then \begin{enumerate}
	\item The $L_\infty$-algebra $L\oplus\h[-1]$ is weakly filtered (so it contains $\MC$ elements).
\item The following isomorphism of sets holds:
	\[	\MC(L\oplus\h[-1])\cong\VMC(L).\]
	\end{enumerate}
\end{lem}
\begin{proof}
	For (1) consider the following filtration on $L\oplus\h[-1]$:
	\[
	\mathcal{F}_1:=L\oplus\h[-1]\supset\mathcal{F}_2:=P([L,\h])[-1]\supset\mathcal{F}_3:=P([[L,\h],\h])[-1]\supset\cdots.
	\]
	By the definition of an admissible $V$-structure, the above filtration is complete; moreover it clearly satisfies condition (2) of Definition (\ref{def:filtered}) with $l=3$.
	
For (2), let $(x[1],h)\in\MC(L\oplus\h[-1])$ so that $x\in L^1$ and $h\in \h^0$. We have
	\begin{eqnarray*}&&\sum_{n=1}^{\infty}\frac{1}{n!}\check{m}_n\Big((x[1],h),\ldots,(x[1],h)\Big)\\
	&=&\check{m}_1(x[1],h)+\frac{1}{2}\check{m}_2\Big((x[1],h),(x[1],h)\Big)+\sum_{n=3}^{\infty}\frac{1}{n!}\check{m}_n\Big((x[1],h),\ldots,(x[1],h)\Big)\\ &=&\Big(-d(x)[1],P(x+d(h))\Big)+\Big(-\frac{1}{2}[x,x][1],P\ad_h (x)+\frac{1}{2}P\ad_h (d(h))\Big)\\
&&+\Big(0,P\sum_{n=3}^{\infty}\frac{1}{(n-1)!}\ad_h^{n-1}(x)\Big)+\Big(0,P\sum_{n=3}^{\infty}\frac{1}{n!}\ad_h^{n-1}(d(h))\Big)
	\end{eqnarray*}
	from which it follows that
 \begin{eqnarray*}
d(x)+\frac{1}{2}[x,x]&=&0,\\
P(x)+\sum_{n=1}^{\infty}\frac{1}{n!}P\Big((\ad_h)^n(x)+\ad_h^{n-1}(dh)\Big)&=&0.
\end{eqnarray*}
  Therefore, $x\in \MC(L)$ and   $P(x*h)=0$. All told, we obtain that $(x,h)\in\VMC(L)$.  The same calculation performed in the reverse order, shows that, conversely, if $(x,h)\in\VMC(L)$, then $(x[1],h)\in\MC(L\oplus\h[-1])$.
\end{proof}	
Furthermore, the following result holds.
\begin{prop}\label{prop:VMC}
	Let $L$ be a dgla with a V-structure. Then the functor $\VMC(L,-)$ is representable. The complete cdga representing it is the representing cdga of the $L_\infty$-algebra $ L\oplus\h[-1]$ constructed in Theorem \ref{thm:db-big-homotopy-lie-algebra}.
\end{prop}
\begin{proof}
	Let $A$ be a complete cdga. Then $(A_{\geq 1}\otimes L,\id\otimes P)$ is a pronilpotent dgla with a V-structure (which is then admissible) and by Lemma \ref{lem:VMC} we have:
	\begin{align*}
	\VMC(L,A)&\cong\VMC(A_{\geq 1}\otimes L)\\
	&\cong \MC(A_{\geq 1}\otimes(L\oplus\h[-1]))\\
	&\cong\MC(L\oplus\h[-1],A)
	\end{align*}
	as required.
\end{proof}

It is clear that the $L_\infty$-algebra  $L\oplus \h[-1]$ defined above, is quasi-isomorphic to the dgla $\ker P$ and thus, there is an $L_\infty$-map $j:L\oplus \h[-1]\to\ker P$ that is homotopy inverse to the inclusion $\ker P\hookrightarrow L\hookrightarrow L\oplus \h[-1]$. It follows from (\ref{eq:extensionvoronov}) that there is a homotopy fibre sequence of $L_\infty$-algebras and maps:
\begin{equation}\label{eq:cofibre}
\h[-1]\to \ker P\to L.
\end{equation}
Denote by $i:\h[-1]\to \ker P$ the corresponding $L_\infty$-map in the above homotopy fibre sequence. If the given $V$-structure is admissible, there is an induced map $\MC(i):\MC(\h[-1])\to \MC(\ker P)$. We will find this map explicitly.

\begin{prop}\label{prop:exponential} Let $L$ be a pronilpotent dgla with a V-structure. Then:
	\begin{enumerate}\item
		The $L_\infty$-map $j: L\oplus \h[-1]\to\ker P$ induces the map $$
		\MC(j):\MC(L\oplus \h[-1])\to\MC(\ker P),
		$$
		so that for $(x[1],h)\in\MC(L\oplus \h[-1])$ it holds that
		\[
		\MC(j)(x[1],h)=x*h.	
		\]
		\item	The $L_\infty$-map $i:\h[-1]\to \ker P$ induces the map
		$$
		\MC(i):\MC(\h[-1])\to\MC(\ker P),
		$$
		so that for $h\in\MC(\h[-1])$ it holds that
		\[
		\MC(i)(h)=0*h:=\sum_{n=1}^{\infty}\frac{1}{n!}\ad_h^{n-1}d(h).	
		\]
	\end{enumerate}
\end{prop}
\begin{proof}
	It is clear that (2) follows from (1). For (1) let $(x[1],h)\in \MC(L\oplus \h[-1])$, i.e. a VMC element where $x\in L^1$ and $h\in \h^0$. Then the element $x*h$ is an MC element in $L$ (gauge equivalent to $x$) and it belongs to $\ker P$ by definition of a VMC element. To finish the proof it suffices to observe that this morphism splits the canonical map $\MC(\ker P)\to \MC(L\oplus \h[-1])$ induced by the inclusion $\ker P\hookrightarrow L\oplus \h[-1]$.
\end{proof}	
\begin{cor}Let $L$   be as in Proposition \ref{prop:exponential}. Then:
	\begin{enumerate}
		\item
		The $L_\infty$-map $j:L\oplus \h[-1]\to\ker P$ has the form:
		\begin{align*}
		\check{j}_1(x[1],h)=&(\id-P)(x)[1],\\
		\check{j}_k(x[1],h_1,\cdots,h_{k-1})=&(\id-P)[\cdots[[x,h_1],h_2]\cdots,h_{k-1}][1],\\
		\check{j}_k(h_1,\ldots,h_k)=&(\id-P)[\cdots[d(h_1),h_2]\cdots,h_{k}][1],\quad k\geq 2.
		\end{align*}
		\item	
		The $L_\infty$-map $i:\h[-1]\to\ker P$ from \eqref{eq:cofibre} has the following form:
		\begin{align*}
		\check{i}_k(h_1,\ldots,h_k)=&(\id-P)[\cdots[d(h_1),h_2]\cdots,h_{k}][1],\quad k\geq 2.
		\end{align*}
	\end{enumerate}	
\end{cor}
\begin{proof}
It suffices to show that the map $f:=(f_1,\ldots, f_k\ldots)$ as defined above induces the correct map on MC elements with values in arbitrary complete cdgas (i.e. it agrees with the formula of Proposition \ref{prop:exponential} (2)). This is a straightforward calculation similar to that of Lemma \ref{lem:VMC}.
\end{proof}	
\begin{example}Here is one of the simplest yet nontrivial examples of the above construction. Let $(V,m)$ be an $L_\infty$-algebra. Set
	$L:=\Der \hat{S}V^*[-1]$ and $P$ be the natural projection $\Der \hat{S}V^*[-1]\to  V[1]$ and $\Delta:=m\in \Derbar \hat{S}V^*[-1]$. Note that $(L,[\Delta,-])$ is nothing but the Chevalley-Eilenberg complex of the $L_\infty$-algebra $V$ whereas the dg Lie subalgebra $\Derbar \hat{S}V^*[-1]$ is its truncation.  Then $\h\cong V[1]$ and the $L_\infty$-structure on $V$ coming from higher derived  brackets is just the original $L_\infty$-structure. This result, obtained by a different method, as well as an explicit $L_\infty$-map
	$V\to \Derbar \hat{S}V^*[-1]$ is contained in \cite{CL}.
\end{example}
\begin{rem}
	The $L_\infty$-algebra $\h[-1]$ is a model for a homotopy fiber of the inclusion of dglas $\ker P\to L$. Our methods can be extended to the case when $\h=\im(P)$ is not abelian, but we will not pursue this route as our applications are only concerned with the abelian case. The non-abelian case was treated in \cite{Bandiera} using different methods.
	
	The functor VMC constructed above is analogous to the functor $\operatorname{Def}_\chi$associated to an arbitrary
	morphism (not necessarily an inclusion) of dglas $\chi:M\to L$ introduced in \cite{Man} in the sense that both can be interpreted as MC functors of homotopy fibers of the corresponding maps.
\end{rem}

\section{Homotopy relative Rota-Baxter Lie algebras}
Let $(\g,m)$ be an $L_\infty$-algebra and $V$ be a representation of $\g$, i.e. a graded  vector space together with an $L_\infty$-map $\rho:\g\to\gl(V)$. Recall that this data can be represented as an MC element
$$
\Phi:=(m, \rho)\in  \huaL_{\HLR}(\g, V):=\Derbar\hat{S}\g^*[-1]\ltimes \big(\hat{S}_{\geq 1}\g^*[-1]\otimes\gl(V)\big).
$$
Via the natural inclusion
$$
\huaL_{\HLR}(\g,V)\subset \Derbar\hat{S}(\g\oplus V)^*[-1]
$$
the element $\Phi$ can be regarded as an MC element of the graded Lie algebra $\Derbar\hat{S}(\g\oplus V)^*[-1]$, cf. Remark \ref{lie-pair-gla}. We set
$$
\h=\Hom(\g^*[-1],\hat{S}_{\geq 1}V^*[-1]).
$$
Then $\h$ is the abelian Lie subalgebra in $\Derbar\hat{S}(\g\oplus V)^*[-1]$ consisting of continuous derivations vanishing on $V^*[-1]$ and mapping $\g^*[-1]$ to $\hat{S}_{\geq 1}V^*[-1]$. Note that the natural direct complement $\h^\perp$ of $\h$ in $\Derbar\hat{S}(\g\oplus V)^*[-1]$ is closed with respect to the commutator bracket and contains $\Phi$. We will define  $P$ to be the projector in $\Derbar\hat{S}(\g\oplus V)^*[-1]$ onto $\h$. Thus, we obtain a $V$-structure as following:

\begin{prop}\label{prop:VoronovRB}
Let $(\g,m)$ be an $L_\infty$-algebra and $(V,\rho)$ be a representation of $\g$. Then the projection onto the subspace $\h$ determines an admissible  $V$-structure on the dgla $\Derbar\hat{S}(\g\oplus V)^*[-1]$ supplied with the commutator bracket and the differential $d=[\Phi,\cdot]$.

Consequently, $(\h[-1],\{\check{m}_k\}_{k=1}^{\infty})$ is a weakly filtered $L_\infty$-algebra, where higher products $\check{m}_k$ are given by formulas \eqref{V-shla}.
\end{prop}
\begin{proof}
We only need to show that the given $V$-structure is admissible. To that end
let $T\in\h\cong\Hom(\g^*[-1],\hat{S}_{\geq 1}V^*[-1])$. The element $T$ can be written as a sum
\begin{equation}\label{eq:RBoperator}
T=T_1+T_2+\cdots,
\end{equation}
where $T_n$ is the  order $n$  part of $T$ so we can write $T_n: \g^*[-1]\to \hat{S}^nV^*[-1]$. The element $T$ is a derivation of $\hat{S}(\g\oplus V)^*[-1]$ but we will view it merely as an endomorphism of $\hat{S}(\g\oplus V)^*[-1]$. It is clear that $T_1$ is a nilpotent endomorphism (in fact $T_1^2=0$) and it follows that the adjoint action of $T$ is pronilpotent.
\end{proof}
\begin{rem}
	In fact, the $L_\infty$-algebra $\h$ is even \emph{filtered}, cf. \cite[Remark 5.9]{LTS}
\end{rem}

We can now give a compact definition of a homotopy relative RB operator.

\begin{defi}
	A {\em homotopy relative $\RB$  operator} on an $L_\infty$-algebra $(\g,m)$ supplied with a representation $(V,\rho)$ is an element
	$T\in\h^0$ such that $P(e^{\ad_T}(m+\rho))=0$ (i.e.  $T$ is an MC element of the  $L_\infty$-algebra $(\h[-1],\{{\check{m}_k}\}_{k=1}^{\infty}))$.
\end{defi}
\begin{rem}
Let $T$ be a homotopy relative $\RB$ operator compatible with an $L_\infty$-algebra $(\g,m)$ and its representation $(V,\rho)$ and $T_n,n=1,2,\cdots$ be its components as in (\ref{eq:RBoperator}).  Consider the dual map of $T_n$, thus we have the degree $0$ map $\check{T}_n:\hat{S}^n V[1]\to  \g[1]$ for $n=1,2,\cdots.$ Moreover, $\check{T}_n$ is graded symmetric and satisfy the homotopy relative $\RB$ relation. See \cite{LTS} for more details about homotopy relative $\RB$ operators.
\end{rem}

\begin{rem}
By the well-known formula for the exponential of the adjoint representation, the homotopy relative  RB condition $P(e^{\ad_T}(m+\rho))=0$ can be rewritten as
$P(e^{-T}(m+\rho)e^T)=0$.
\end{rem}
\begin{rem}  In the classical case for $(\g,[-,-]_\g)$ is an ordinary Lie algebra and $\rho:\g\to\gl(V)$ is a representation of $\g$ on $V$. For any $T\in\Hom(V,\g)$, we have
$$
e^{\ad_T}(m+\rho) =e^{-T}\circ (m+\rho) \circ (e^{T}\otimes e^{T}).
$$
  For all $u,v\in V$, we have
\begin{eqnarray*}
\big(e^{-T}\circ (m+\rho) \circ (e^{T}\otimes e^{T})\big)(u,v)&=&(1-T)(m+\rho)(u+Tu,v+Tv)\\
 &=&(1-T)([Tu,Tv]_\g+\rho(Tu)v-\rho(Tv)u)\\
 &=&\rho(Tu)v-\rho(Tv)u+[Tu,Tv]_\g-T(\rho(Tu)v-\rho(Tv)u).
\end{eqnarray*}
Therefore, $P(e^{\ad_T}(m+\rho))=0$ will give us
$$
[Tu,Tv]_\g=T(\rho(Tu)v-\rho(Tv)u),
$$
which implies that $T$ is a {\em relative $\RB$ operator} (also called an $\mathcal{O}$-operator) on the Lie algebra $(\g,m)$ with respect to the representation $(V,\rho)$. See \cite{Bai} for more details.
\end{rem}

\begin{defi}
Let $\g$ and $V$ be graded  vector spaces. A {\em homotopy relative $\RB$ Lie algebra} on $\g$ and $V$ is a triple $ (m,\rho,T)$, where $ m$ is an $L_\infty$-algebra structure on $\g$, $\rho$ is a representation of $(\g,m)$ on   $V$ and $T$ is a  homotopy relative $\RB$ operator compatible the $L_\infty$-algebra $(\g,m)$ and its representation $(V,\rho)$.
\end{defi}

Let $\g$ and $V$ be graded vector spaces.
Since $\huaL_{\HLR}(\g,V)$ is a graded Lie subalgebra of $\Derbar\hat{S}(\g\oplus V)^*[-1]$, we have the following corollary:

\begin{cor}\label{cor:Linfty}
 With the above notation,  $(\huaL_{\HLR}(\g,V)\oplus \h[-1],\{\check{m}_i\}_{i=1}^{\infty})$ is an $L_\infty$-algebra, where $\check{m}_i$ are given by
 \begin{eqnarray*}
  \check{m}_2(Q[1],Q'[1])&=&(-1)^{|Q|}[Q,Q'][1], \\
\check{m}_k(Q[1],\theta_1,\cdots,\theta_{k-1})&=&P[\cdots[Q,\theta_1],\cdots,\theta_{k-1}],
\end{eqnarray*}
for homogeneous elements   $\theta_1,\cdots,\theta_{k-1}\in \h$, homogeneous elements  $Q,Q'\in \huaL$, and all the other higher products vanish, unless they can be obtained from the ones listed by permutations of arguments. We denote the $L_\infty$-algebra $(\huaL_{\HLR}(\g,V)\oplus \h[-1],\{\check{m}_i\}_{i=1}^{\infty})$ by $\huaL_\LHRB(\g,V)$.
\end{cor}
\begin{proof}
  It follows from Theorem \ref{thm:db-big-homotopy-lie-algebra}, Remark \ref{thm:db-big-homotopy-lie-algebra-small} and Proposition \ref{prop:VoronovRB}.
\end{proof}

\emptycomment{
Moreover, for all $n\ge 1$, we set
\begin{eqnarray*}\label{filtration-homotopy-rota}
\quad\huaF_n(\huaL_{\HLR}[1]\oplus \h)=\huaL_{\HLR}[1]\oplus\Big(\Pi_{i=n}^{\infty}\Hom(\g^*[-1],\hat{S}^iV^*[-1])\Big).
\end{eqnarray*}

\begin{lem}\label{filtered-homotopy-lie-homotopy-rota}
With above notation, $(\huaL_{\HLR}[1]\oplus \h,\{l_i\}_{i=1}^{\infty})$ is a filtered $L_\infty $-algebra.
\end{lem}
}
Now we show that the $L_\infty$-algebra $\huaL_\LHRB(\g,V)$ governs homotopy relative RB Lie algebras.
\begin{theorem}\label{deformation-rota-baxter}
  Let $\g$ and $V$ be   graded vector spaces,  $m\in\Derbar^1\hat{S}\g^*[-1],~\rho\in\big(\hat{S}_{\geq 1}\g^*[-1]\otimes\gl(V)\big)^1$ and $T\in\Hom^0(\g^*[-1],\hat{S}_{\geq 1}V^*[-1])$. Then the triple $(m,\rho,T)$ is a homotopy relative $\RB$ Lie algebra structure on $\g$ and $V$ if and only if  $(\Phi[1],T)$ is an $\MC$ element of the  $L_\infty$-algebra $\huaL_\LHRB(\g,V)$ given in Corollary \ref{cor:Linfty}, where $\Phi=(m,\rho)$.
\end{theorem}
\begin{proof}
Let $(\Phi[1],T)$ be an $\MC$ element of the  $L_\infty$-algebra $\huaL_\LHRB(\g,V)$. Using the inclusion of $L_\infty$-algebras
\[(\huaL_{\HLR}(\g,V)\oplus \h[-1],\{\check{m}_i\}_{i=1}^{\infty})\subset \Derbar\hat{S}(\g\oplus V)^*[-1]\oplus\h[-1],\{\check{m}_k\}_{k=1}^{\infty}),\]
the pair $(\Phi[1],T)$ can be viewed as an MC element in $\Derbar\hat{S}(\g\oplus V)^*[-1]\oplus\h[-1],\{\check{m}_k\}_{k=1}^{\infty})$. This implies, by Lemma \ref{lem:VMC}, that $P(e^{-T}(m+\rho)e^T)=0$, i.e.   $T$ is a homotopy relative RB operator on $\g$ with respect to the representation $\rho$. Conversely, given a homotopy relative RB operator $T$, the same argument traced in the reverse order, shows that the pair $(\Phi[1],T)$ is an $\MC$ element of the  $L_\infty$-algebra $\huaL_\LHRB(\g,V)$.

\end{proof}

\begin{rem}
	It is clear that the $L_\infty$-algebra $\huaL_\LHRB$ is included in an $L_\infty$-extension
	\[
	\h[-1]\to \huaL_\LHRB(\g,V)\to \huaL_{\HLR}(\g,V)
	\]
	where 	$\h[-1]$ has the trivial $L_\infty$-structure.
\end{rem}

\begin{defi}\label{A-linear-homotopy-lie-rota-baxter}
Let $A$ be a complete cdga, $\g$ and $V$ be  graded vector spaces. Then an {\em $A$-linear homotopy relative $\RB$ Lie algebra} structure on $A\otimes \g$ and $A\otimes V$ is a pair $(\Phi_A[1],T_A)$, which is an MC element of  the $L_\infty$-algebra $A_{\geq 1}\otimes \huaL_\LHRB(\g,V)$, where $\Phi_A=(m_A,\rho_A)$.
\end{defi}

Let $\huaF_{\HRB}$ be the functor associating to a complete cdga $A$ the set of  $A$-linear  homotopy relative $\RB$ Lie algebra structures on $A\otimes\g$ and $A\otimes V$. Then we have the following result.

\begin{theorem}\label{thm:RBrep}
 The functor  $\huaF_{\HRB}$ is represented by the complete cdga $\hat{S}\huaL^*_\LHRB(\g,V)[-1]$.
\end{theorem}

\begin{proof}
By the definition of an $A$-linear homotopy relative RB Lie algebra   and Theorem \ref{deformation-rota-baxter}, we deduce that $\huaF_{\HRB}(A)=\MC(\huaL_\LHRB(\g,V),A)$. Therefore, by Theorem \ref{representable-thm}, we obtain that the functor $\huaF_{\HRB}$ is represented by the complete cdga $\hat{S}\huaL^*_\LHRB(\g,V)[-1]$.
\end{proof}

\section{Shifted Poisson algebras, $r_\infty$-matrices and triangular $L_\infty$-bialgebras}

In this section, we describe a certain doubling construction for shifted Poisson algebras. Our exposition is a straightforward  modification of the corresponding $\mathbb Z/2$-graded construction in \cite{BLaz}. The construction of higher derived brackets applied to shifted Poisson algebras, gives rise to higher Schouten Lie algebras, whose Maurer-Cartan elements correspond to  $r_\infty$-matrices. Then we show that $r_\infty$-matrices give rise to triangular $L_\infty$-bialgebras.

\subsection{Doubling construction for shifted Poisson algebras}
Let $\g$ be a graded vector space, here and later on assumed to be finite-dimensional. Then the $n$-shifted cotangent bundle $T^*[n]\g[1]$ is a graded symplectic manifold equipped with a degree $n$ symplectic structure. Consequently, the  pairing of degree $-n$
$$\g^*[-1]\otimes \g[1-n]\to \ground$$ determines  an $n$-shifted Poisson algebra structure on the complete pseudocompact algebra $\hat{S}(\g^*[-1]\oplus   \g[1-n])$. The corresponding Poisson bracket $\{-,-\}$ can be viewed as a graded Lie algebra structure on $(\hat{S}(\g^*[-1]\oplus \g[1-n]))[n]$, in other words, it is a degree zero map
\[\{-,-\}:  (\hat{S}(\g^*[-1]\oplus \g[1-n]))[n]\otimes  (\hat{S}\g^*[-1]\oplus \g[1-n]))[n]\to  (\hat{S}(\g^*[-1]\oplus \g[1-n]))[n]\]
satisfying the graded Jacobi identity.

Later on we will also need to work with the graded vector space
$$\hat{S}^\prime(\g^*[-1]\oplus \g[1-n]):=\hat{S}_{\geq 1}\g^*[-1]\otimes\hat S_{\geq 1}(\g[1-n]),$$ which is clearly a (shifted) Poisson subalgebra of $\hat{S}\g^*[-1]\otimes\hat{S}(\g[1-n])$.

Consider the graded Lie algebra $\Derbar\hat{S} \g^*[-1]$  with respect to the commutator bracket.
Since any derivation is determined by its value on $ \g^*[-1]$, we can identify $\Derbar\hat{S}\g^*[-1]$ with the graded vector space \[\Hom(\g^*[-1],\hat{S}\g^*[-1])\cong \hat{S}\g^*[-1]\otimes \g[1].\]

\begin{defi}
	The {\em $n$-shifted double} is the map 	
	\[D_n:\Derbar\hat{S}\g^*[-1]\cong \hat{S} \g^*[-1]\otimes \g[1]\rightarrow (\hat{S}\g^*[-1])[n]\otimes \hat{S}_{\geq 1}  (\g[1-n])\subset
	(\hat{S}_{\geq 1}(\g^*[-1]\oplus \g[1-n]))[n] \]
	given by the formula
\[\hat{S} \g^*[-1]\otimes \g[1]\supset f\otimes w\mapsto (-1)^{n|f|} f[n]\otimes(w[-n]) \in(\hat{S}\g^*[-1])[n]\otimes \hat{S}_{\geq 1}  (\g[1-n]).\]
\end{defi}

\emptycomment{
 {\color{red}Andrey: First, I just noticed that there are some typos in the above. For example $
	(\hat{S}_{\geq 1}(\g^*[-1]\otimes \g[1-n]))[n]$ should be $
(	\hat{S}_{\geq 1}(\g^*[-1]\oplus \g[1-n]))[n]$. I am sending along my original note for comparison (page 7). Second, I am confused by your suggestion. Surely, there are two types of double: even and odd (working in the $\mathbb{Z}/2$-graded context) or there are $n$-doubles if you work in $\mathbb{Z}$-graded spaces. You cannot reduce a $1$-shifted double to a $0$-shifted double by a shift -- these are different structures and they have different properties. Similarly there are $n$-shifted $L_\infty$ bialgebras and you cannot reduce a $0$-shifted $L_\infty$ bialgebra to a $1$-shifted $L_\infty$ bialgebra. So we must consider $n$-shifted structures if we want to cover all cases.

Note also, that in the original note the domain of the $n$-shifted double was $\Der \hat{S}W^*$, without any shifts. I think this is how it should be defined, and then we take $W=\g[1]$.}}

Then we have the following result.
\begin{prop}\label{prop:D}
	The map $D_n$ is a map of graded Lie algebras.
\end{prop}
\begin{proof}
	The proof of\cite[Theorem 3.2]{BLaz} carries over with only notational modifications. The cases when $n$ is odd or even are slightly different.
\end{proof}	

\subsection{$r_\infty$-matrices and triangular $L_\infty$-bialgebras}
Let $(\g,m)$ be an $L_\infty$-algebra. In this section we define, using the doubling construction and Voronov's higher derived brackets, the higher shifted Schouten Lie algebra and an $n$-shifted $r_\infty$-matrix for $\g$. This leads naturally to the notion of an $n$-shifted triangular $L_\infty$-bialgebra.  Ordinary (or $0$-shifted in our terminology) $L_\infty$-bialgebras were introduced in \cite{Kra} and its shifted version was considered in \cite{Mer, BSZ}.

  Consider the graded Lie algebra
\[
(\hat{S}(\g^*[-1]\oplus \g[1-n]))[n]
\]
where the graded Lie bracket was defined in the previous subsection.
The $L_\infty$-algebra structure $m$ is an MC element in $\Derbar\hat{S}\g^*[-1]$ and so, $D_n(m)$ is an MC element in the graded Lie algebra
$(\hat{S}(\g^*[-1]\oplus \g[1-n]))[n]$ making it a dgla.
Note that $ (\hat{S}\g[1-n])[n]\subset (\hat{S}(\g^*[-1]\oplus \g[1-n]))[n]$ is an abelian Lie subalgebra whose direct complement
is 	a dg Lie subalgebra. Then, Voronov's derived brackets construction implies the following result.

\begin{theorem}\label{thm:dbr}
  Let $(\g,m)$ be an $L_\infty$-algebra. Then the differential graded Lie algebra $$\Big(L:=(\hat{S}(\g^*[-1]\oplus \g[1-n]))[n],\{-,-\},d=\{D_n(m),\cdot\}\Big)$$ is endowed with an admissible V-structure  $P:L\to L $, which is the projection to $\h:=(\hat{S}_{\geq 2}\g[1-n])[n]$. Consequently, $(\h[-1],\{\check{m}_k\}_{k=1}^{\infty})$ is an $L_\infty$-algebra, where $\check{m}_k$ is given by \eqref{V-shla}.

Moreover, there is an $L_\infty$-algebra structure on $L\oplus \h[-1]$ given by
  \begin{equation}\label{rV-shla-big-algebra}\left\{\begin{array}{rcl}
 \check{m}_1(x[1],h)&=&(-d(x)[1],P(x+d(h))),\\
 \check{m}_2(x[1],y[1])&=&(-1)^{|x|}\{x,y\}[1],\\
 \check{m}_k(x[1],h_1,\cdots,h_{k-1})&=&P\{\cdots\{\{x,h_1\},h_2\},\cdots,h_{k-1}\},\quad k\geq 2,\\
 \check{m}_k(h_1,\cdots,h_{k-1},h_k)&=&P\{\cdots\{d(h_1),h_2\},\cdots,h_{k}\},\quad k\geq 2,
\end{array}\right.
\end{equation}
for all $x,y\in L$ and $h,h_1,\ldots, h_k\in\h$. The remaining  $L_\infty$-products vanish, unless they are obtained from those above by permutations of arguments. Moreover there exists
 the following  $L_\infty$-extension:
	\begin{equation}\label{eq:extensiondouble} (\hat{S}_{\geq 2}\g[1-n])[n-1] \to (\hat{S}(\g^*[-1]\oplus \g[1-n]))[n] \oplus (\hat{S}_{\geq 2}\g[1-n])[n-1]\to (\hat{S}(\g^*[-1]\oplus \g[1-n]))[n].
 \end{equation}
\end{theorem}
\begin{proof}
Arguing as in the proof of Proposition \ref{prop:VoronovRB}, let $h\in\h$; then $h=h_1+h_2+\cdots$ where $h_n\in (S_{\geq n+1}\g[1-n])[n-1]$. In particular, the element $h_1\in (S_{\geq 2}\g[1-n])[n-1]$ viewed as an endomorphism of $L$ is nilpotent (even has square zero) and we conclude that the specified $V$-structure is admissible.  The stated formulas for the $L_\infty$-products follow from Theorem \ref{thm:db-big-homotopy-lie-algebra}.
\end{proof}
\begin{rem}
	If we choose $(\hat{S}\g[1-n])[n]$ as the abelian subalgebra in $L$, then the corresponding V-structure will not be admissible. Nevertheless, the derived brackets construction is applicable and $(\hat{S}\g[1-n])[n-1]$ becomes an $L_\infty$-algebra where formulas \eqref{rV-shla-big-algebra} still hold. Clearly,
	$\h=(\hat{S}_{\geq 2}\g[1-n])[n-1]$ is an
	$L_\infty$-subalgebra in  $(\hat{S}\g[1-n])[n-1]$.
\end{rem}
\begin{defi}
The $L_\infty$-algebra $(\hat{S}\g[1-n])[n-1],\{{\check{m}_k}\}_{k=1}^{\infty})$ is called the  {\em higher Schouten Lie algebra} of the $L_\infty$-algebra $(\g,m)$.
\end{defi}

\begin{rem}
	It is well known that for an ordinary Lie algebra  $\g$, the exterior algebra $\Lambda(\g):=\oplus_{k=0}^{\infty}\wedge^{k+1}\g$ is a graded Lie algebra ($\wedge^{k+1}\g$ is of degree $k$), which is also called the Schouten Lie algebra.
 Now the $L_\infty$-algebra structure on $(\hat{S}\g[1-n])[n-1]$ given in Theorem \ref{thm:dbr} only contains ${m_2}$. If, furthermore, $n=2$, it reduces to  the Schouten Lie algebra, i.e. $$(\hat{S}\g[-1])[1]=\Lambda(\g).$$

If, on the other hand, $n=1$, we obtain a graded Lie algebra structure on $\hat{S} \g$; this is the well-known Poisson bracket on the completion of $\operatorname{gr}U\g$, the associated graded to the universal enveloping algebra of $\g$ (which is isomorphic, by the Poincare-Birkhoff-Witt theorem, to $S\g$).
\end{rem}

We can now give the definition of an $n$-shifted $r_\infty$-matrix.

\begin{defi}\label{def:r-infty}	Let  $(\g,m)$ be an $L_\infty$-algebra. \begin{enumerate}
		\item An {\em$n$-shifted $r_\infty$-matrix for $\g$} is a degree $0$ element $r\in (\hat{S}_{\geq 2}\g[1-n])[n]$ such that $P(e^{\ad_r}D_n(m))=0$, i.e. $r$ is an MC element in the $L_\infty$-algebra $(\hat{S}_{\geq 2}\g[1-n])[n-1]$.
		\item
	Let $A$ be a complete cdga. Then an {\em$n$-shifted $r_\infty$-matrix for $\g$ with values in $A$} is a degree $0$ element $r_A\in A_{\geq 1}\otimes(\hat{S}_{\geq 2}\g[1-n])[n]$ such that $P(e^{\ad_{r_A}}D_n(m))=0$, i.e. $r_A$ is an MC element in the $L_\infty$-algebra $(\hat{S}_{\geq 2}\g[1-n])[n-1]$ with values in $A$.
\end{enumerate}
\end{defi}

\begin{rem}
It is possible to define an $r_\infty$-matrix using the $L_\infty$-algebra  $(\hat{S}\g[1-n])[n-1]$ as opposed to the smaller $L_\infty$-algebra $(\hat{S}_{\geq 2}\g[1-n])[n-1]$; however since the former does not contain MC elements, it is then necessary to introduce an auxiliary complete cdga.	Such a definition in the 0-shifted case and taking $\ground[[\lambda]]$ as a complete cdga, was given in  \cite{BVor} using a different method. As far as we know, $n$-shifted $r_\infty$-matrices for $n$ odd have not been considered, even in the case of ordinary (graded) Lie algebras. Our definition is closer to the classical notion of an $r$-matrix, \cite{STS} and specializes to this notion in the case when the formal power series $r\in (\hat{S}_{\geq 2}\g[1-n])[n-1]$ contains only the first, quadratic, term.
\end{rem}

\emptycomment{
The notion of a (shifted) $r_\infty$-matrix for an $L_\infty$-algebra $(\g,m)$ can be made $A$-linear, where $A$ is a complete cdga; namely an $n$-shifted $A$-linear $r_\infty$-matrix is an MC element in the $L_\infty$-algebra $A_{\geq 1}\otimes(\hat{S}\g[1-n])[n-1]$.  Given a graded vector space $\g$, consider the functor $R_\infty$ associating to a complete cdga $A$ the pair $(m,r_A)$ where $m$ is an $L_\infty$-algebra structure on $\g$ and $r_A$ is an $n$-shifted $A$-linear $r_\infty$-matrix for $(\g,m)$. Then we have the following result.
}

\begin{defi}
	The structure of an {\em $n$-shifted $L_\infty$-bialgebra} on a graded vector space $\g$ is an MC element in $\hat{S}^\prime(\g^*[-1]\oplus \g[1-n])[n]$, i.e. an element $h\in\hat{S}^\prime(\g^*[-1]\oplus \g[1-n])[n]$ of degree 1 such that $\{h,h\}=0$.
\end{defi}

\begin{rem}
	The projections
	\[\hat{S}^\prime(\g^*[-1]\oplus \g[1-n])[n]\to 	\g[1] \otimes\hat{S}\g^*[-1]\cong\Der\hat{S}\g^*[-1]\]
	and
	\[
	\hat{S}^\prime(\g^*[-1]\oplus \g[1-n])[n]\to ( \hat{S}\g[1-n])[n]\otimes\g^*[-1]\cong\Der\hat{S}(\g[1-n])\cong\Der\hat{S} (\g^*[n-2])^*[-1]
	\]
	are graded Lie algebra maps. It follows that an $n$-shifted $L_\infty$-bialgebra structure on $\g$ determines
	\begin{itemize}
\item an $L_\infty$-algebra structure on $\g$ and
\item an $L_\infty$-algebra structure on $\g^*[n-2]$ (equivalently, an $L_\infty$-\emph{coalgebra} structure on $\g[2-n]$).
	\end{itemize}
Additionally, there are appropriate compatibility conditions between these structures. This explains the terminology `shifted $L_\infty$-bialgebra'.
	
	Note also that the Poisson algebra $\hat{S}_{\geq 1}(\g^*[-1]\oplus \g[1-n])[n]$ (as opposed to its Poisson subalgebra $\hat{S}^\prime(\g^*[-1]\oplus \g[1-n])[n]$) leads to the notion of an $L_\infty$-quasi-bialgebra.
\end{rem}

Classically, an $r$-matrix in a Lie algebra $\g$ gives rise to a Lie bialgebra structure on $\g$, called \emph{triangular}. We will now formulate a higher version of this result.
\begin{theorem}\label{th:rbialgebra}
	Let  $(\g,m)$  be an $L_\infty$-algebra and $r\in(\hat{S}_{\geq 2}\g[1-n])[n]$ be an $n$-shifted $r_\infty$-matrix.  Define   $r(m)\in (\hat{S}(\g^*[-1]\oplus \g[1-n]))[n]$ by
\begin{equation}\label{eq:triangular} r(m):=e^{\ad_r} D_n(m).\end{equation}
 Then $r(m)$ is an $\MC$ element in $(\hat{S}'(\g^*[-1]\oplus \g[1-n]))[n]$  and   determines the structure of an $n$-shifted $L_\infty$-bialgebra on $\g$.
\end{theorem}
\begin{proof}
Note that $D_n(m)$ is an MC element in the graded Lie algebra $(\hat{S}(\g^*[-1]\oplus \g[1-n]))[n]$ as the image of the MC element $m$ under the gla map
$D_n$, cf. Proposition \ref{prop:D}. It follows that $e^{\ad_r} D_n(m)$ is likewise an MC element as obtained from $D_n(m)$ by a gauge transformation. 	Next, since $r$ is an $r_\infty$-matrix, we have $P(e^{\ad_r} D_n(m))=0$ and it follows that $e^{\ad_r} D_n(m)\in \hat{S}'(\g^*[-1]\oplus \g[1-n])[n]$ and we are done.
\end{proof}

\begin{defi}\label{def:rinfty}
	Given an $L_\infty$-algebra $(\g,m)$ and an $n$-shifted $r_\infty$-matrix $r$, the $n$-shifted $L_\infty$-bialgebra $r(m)$ defined by formula (\ref{eq:triangular}) will be called a \emph{triangular $n$-shifted $L_\infty$-bialgebra}. By abuse of terminology, we will also refer to the triple $(\g,m,r)$ as a \emph{triangular $n$-shifted $L_\infty$-bialgebra}.
\end{defi}
\begin{rem}\begin{enumerate}\item
		The statement of Theorem \ref{th:rbialgebra} for $n=2$ and in the context of $\ground[[\lambda]]$-linear $r_\infty$-matrices is essentially the main result of \cite{BVor}, cf. Theorem 2 and the discussion following it in op. cit.
\item For an ordinary Lie algebra,  Theorem \ref{th:rbialgebra} reduces to the standard construction of a triangular Lie bialgebra.
	\end{enumerate}
\end{rem}

Let $\g$ be a vector space;recall that $\h\cong(\hat{S}_{\geq 2}\g[1-n])[n]$. Since $\Derbar\hat{S}\g^*[-1]$ is a graded Lie subalgebra of $(\hat{S}(\g^*[-1]\oplus \g[1-n]))[n]$, we have the following corollary:

\begin{cor}\label{cor:Linfty-Matrix}
 With the above notation,  $(\Derbar\hat{S}\g^*[-1]\oplus \h[-1],\{\check{m}_i\}_{i=1}^{\infty})$ is an $L_\infty$-algebra, where $l_i$ are given by
 \begin{eqnarray*}
 \check{m}_2(Q[1],Q'[1])&=&(-1)^{|Q|}[Q,Q'][1], \\
\check{m}_k(Q[1],\theta_1,\cdots,\theta_{k-1})&=&P[\cdots[Q,\theta_1],\cdots,\theta_{k-1}],
\end{eqnarray*}
for homogeneous elements   $\theta_1,\cdots,\theta_{k-1}\in \h$, homogeneous elements  $Q,Q'\in \Derbar\hat{S}\g^*[-1]$, and all the other possible combinations vanish.
\end{cor}
\begin{proof}
  It follows from Remark \ref{thm:db-big-homotopy-lie-algebra-small} and Theorem  \ref{thm:dbr}.
\end{proof}

We denote the above  $L_\infty$-algebra $(\Derbar\hat{S}\g^*[-1]\oplus \h[-1],\{\check{m}_i\}_{i=1}^{\infty})$ by $\huaL_{\LHM}(\g)$. We will see that it governs triangular $L_\infty$-bialgebras; the proof is formally analogous to that of Theorem \ref{deformation-rota-baxter}.
\begin{theorem}\label{th:triangular_bialgebra}
	Let $\g$ be a graded vector space,  $m\in\Derbar^1\hat{S}\g^*[-1]$ and $r\in(\hat{S}_{\geq 2}\g[1-n])[n]$. Then the pair $(m[1],r)$ is a triangular $L_\infty$-bialgebra structure on $\g$ if and only if  $(m[1],r)$ is an $\MC$ element of the  $L_\infty$-algebra $\huaL_{\LHM}(\g)$ given in Corollary \ref{cor:Linfty-Matrix}.
\end{theorem}
\begin{proof}
Let $(m[1],r)$ be an $\MC$ element of the  $L_\infty$-algebra $\huaL_{\LHM}(\g)$. Using the inclusion of $L_\infty$-algebras
\[\Derbar\hat{S}\g^*[-1]\oplus \h[-1]\subset (\hat{S}(\g^*[-1]\oplus \g[1-n]))[n]\oplus\h[-1]\]
the pair $(m[1],r)$ can be viewed as an MC element in $(\hat{S}(\g^*[-1]\oplus \g[1-n]))[n]\oplus\h[-1]$. This implies, by Lemma \ref{lem:VMC}, that $P(e^{\ad_r}D_n(m))=0$, i.e.   $r$ is an $r_\infty$-matrix for $\g$. Conversely, given an $r_\infty$-matrix $r$, the same argument traced in the reverse order, shows that the pair $(m[1],r)$ is an $\MC$ element of the  $L_\infty$-algebra $\huaL_{\LHM}(\g)$.	
\end{proof}	
\begin{rem}\label{rem:map}
	It follows from the above that there is a commutative diagram where horizontal arrows are $L_\infty$-extensions and vertical arrows are strict $L_\infty$-maps or graded Lie algebra maps:
	\[
	\xymatrix{
		(\hat{S}\g[1-n])[n-1]\ar[r]\ar@{=}[d]& \huaL_{\LHM}(\g)\ar[d]\ar[r]& \Derbar\hat{S}\g^*[-1]\ar^{D_n}[d]\\
		(\hat{S}\g[1-n])[n-1]\ar[r]\ar@{=}&K\ar[r]&(\hat{S}(\g^*[-1]\oplus \g[1-n]))[n].
	}	
	\]
and $K=(\hat{S}(\g^*[-1]\oplus\g[1-n]))[n] \oplus (\hat{S}\g[1-n])[n-1]$.

\end{rem}

\begin{defi}\label{A-linear-Matrix pair}
Let $A$ be a complete cdga and $\g$ be a graded vector space. Then an {\em $A$-linear triangular $n$-shifted $L_\infty$-bialgebra} structure on $A\otimes \g$  is a pair $(m_A[1],r_A)$, which is an MC element in  the $L_\infty$-algebra $A_{\geq 1}\otimes \huaL_{\LHM}(\g)$.
\end{defi}

Let $\g$ be a graded vector space and $\huaF_{\LHM}$ be the functor associating to a complete cdga $A$ the set of $A$-linear triangular $n$-shifted $L_\infty$-bialgebra structures on $A\otimes\g$. Then we have the following result.

\begin{theorem}\label{thm:Matrix pair}
The functor $\huaF_{\LHM}$ is represented by the complete cdga $\hat{S}\huaL^*_{\LHM}(\g)[-1]$.
\end{theorem}

\begin{proof}
By the definition of an $A$-linear triangular $n$-shifted $L_\infty$-bialgebra    and Theorem \ref{th:triangular_bialgebra}, we deduce that $\huaF_{\LHM}(A)=\MC(\huaL_{\LHM}(\g),A)$. Therefore, by Theorem \ref{representable-thm}, we obtain that the functor $\huaF_{\LHM}$ is represented by the complete cdga $\hat{S}\huaL^*_{\LHM}(\g)[-1]$.
\end{proof}

\emptycomment{
\begin{theorem}
 Let $\g$ be a graded vector space. Then the graded Lie algebra $$\Big(L:=(\hat{S}(\g^*[-1]\oplus \g[1-n]))[n],\{-,-\}\Big)$$ is endowed with a V-structure, where  $P:L\to L $ is the projection to $\h:=(\hat{S}\g[1-n])[n]$ and $\Delta=0$.
Consequently, there is an $L_\infty[1]$-algebra structure on $L[1]\oplus \h$ given by the Voronov's derived bracket construction. Moreover, since $\Derbar\hat{S}\g^*[-1]$ is a subalgebra of $L$, we obtain the $L_\infty[1]$-algebra structure on $\Derbar\hat{S}\g^*[-1][1]\oplus \h$ given by
  \begin{equation}\label{rrV-shla-big-algebra}\left\{\begin{array}{rcl}
  \check{m}_2(x[1],y[1])&=&(-1)^x\{x,y\}[1],\\
 \check{m}_k(x[1],h_1,\cdots,h_{k-1})&=&P\{\cdots\{\{x,h_1\},h_2\},\cdots,h_{k-1}\},\quad k\geq 2,
\end{array}\right.
\end{equation}
for all $x,y\in \Derbar\hat{S}\g^*[-1]$ and $h,h_1,\ldots, h_k\in\h.$

	Furthermore, the functor $R_\infty$ is represented by the $L_\infty$-algebra
$$R_\infty[\g]:= \Derbar\hat{S}\g^*[-1]\oplus (\hat{S}\g[1-n])[n-1].$$
	\end{theorem}

\begin{proof}
  It is straightforward to see that the above $P$ and $\Delta$ endow the graded Lie algebra $(L=(\hat{S}(\g^*[-1]\oplus \g[1-n]))[n],\{-,-\})$ with a V-structure. By Theorem \ref{thm:Voronov}, we obtain the above $L_\infty[1]$-algebra structure \eqref{rrV-shla-big-algebra} on $\Derbar\hat{S}\g^*[-1][1]\oplus \h$.

\end{proof}
}

\section{From triangular $L_\infty$-bialgebras to homotopy relative Rota-Baxter Lie algebras}

In this section, we establish the relation between the $L_\infty$-algebra governing triangular $L_\infty$-bialgebras and the $L_\infty$-algebra governing homotopy relative RB Lie algebras.

Consider again the graded Lie algebra $$\hat{S}\g^*[-1]\otimes (\hat{S}  \g[1-n])[n]\cong\hat{S}(\g^*[-1]\oplus\g[1-n])[n]$$ together with its Poisson bracket $\{-,-\}$. An element in it (a hamiltonian) can be viewed
as a hamiltonian (formal) vector field on $\g^*[-1]\oplus \g[1-n]$, and Poisson brackets correspond to commutators of vector fields. Thus, we have an inclusion of graded Lie algebras
\[
H:\hat{S}_{\geq 1}(\g^*[-1]\oplus\g[1-n])[n]\hookrightarrow \Der\hat{S}(\g^*[-1]\oplus\g[1-n])\cong \Der\hat{S}\big((\g\oplus\g^*[n-2])^*[-1]\big).
\]
MC elements in $\hat{S}_{\geq 2}(\g^*[-1]\oplus\g[1-n])[n]$ correspond, under this map, to elements in the graded Lie algebra \[\Derbar\hat{S}\big((\g\oplus\g^*[n-2])^*[-1]\big)\subset \Der\hat{S}\big((\g\oplus\g^*[n-2])^*[-1]\big)\] that are cyclic $L_\infty$-algebra structures on $\g\oplus\g^*[n-2]$.

What is important for us, is that the image of the abelian Lie subalgebra
\[ (\hat{S}_{\geq 2}\g[1-n])[n]\subset \hat{S}_{\geq 2}(\g^*[-1]\oplus\g[1-n])[n]\]
under the map $H$ is inside the abelian Lie subalgebra in 	$\Derbar\hat{S}\big((\g\oplus\g^*[n-2])^*[-1]\big)$ having the form
\[\Hom(\g^*[-1],\hat{S}(\g^*[n-1])^*)\cong \Hom(\g^*[-1],\hat{S}\g[1-n])\cong \Hom(\hat{S}\g^*[n-1], \g[1]).\]
This is precisely the abelian  subalgebra that was used to construct (using derived brackets) the $L_\infty$-algebra controlling homotopy relative RB Lie algebras given in Theorem \ref{deformation-rota-baxter}.

Taking into account Remark \ref{rem:map}, we obtain the following commutative diagram whose rows are $L_\infty$-extensions:
\[
\xymatrix{
	 (\hat{S}_{\geq 2}\g[1-n])[n-1]\ar[r]\ar@{=}[d]& \huaL_{\LHM}(\g)\ar[d]\ar[r]& \Derbar\hat{S}\g^*[-1]\ar^{D_n}[d]\\
	 (\hat{S}_{\geq 2}\g[1-n])[n-1]\ar[r]\ar[d]^H&  K\ar[r]\ar[d]&\hat{S}_{\geq 2}(\g^*[-1]\oplus\g[1-n])[n]\ar^H[d]\\
	\Hom(\g^*[-1],\hat{S}_{\geq 1}\g[1-n])\ar[r]& X\ar[r]& \Derbar\hat{S}(\g^*[-1]\oplus\g[1-n])
}	
\]
where $K=\hat{S}_{\geq 2}(\g^*[-1]\oplus\g[1-n])[n] \oplus (\hat{S}\g[1-n])[n-1]$ and $X$ is the $L_\infty$-extension obtained from the graded Lie algebra $\Derbar\hat{S}(\g^*[-1]\oplus\g[1-n])$ and its abelian subalgebra $\Hom(\g^*[-1],\hat{S}_{\geq 1}\g[1-n])$ by the derived brackets construction. The vertical maps in the above diagram are dgla maps or strict $L_\infty$-maps.

Now recall that there is an $L_\infty$-extension
\[
 \Hom(\g^*[-1],\hat{S}_{\geq 1}\g[1-n])\to \huaL_{\HRB}(\g,\g^*[n-2]) \to \Derbar\hat{S}\g^*[-1]\ltimes \hat{S}_{\geq 1}\g^*[-1]\otimes\gl(\g^*)
\]
where the $L_\infty$-algebra $\huaL_{\HRB}(\g,\g^*[n-2])$ controls homotopy relative RB Lie algebras on the graded vector space $\g$ with a representation on $\g^*[n-2]$, cf. Corollary \ref{cor:Linfty} (note that $\gl(\g^*[n-2])$ is canonically isomorphic to $\gl(\g^*)$).

Noting that the image of $\Derbar\hat{S}\g^*[-1]$ inside  $\Derbar\hat{S} \big((\g\oplus \g^*[n-2])^*[-1]\big)$ under the map $H\circ D_n$ is contained in the
Lie subalgebra $\huaL_{\HLR}(\g,\g^*[n-2]):=\Derbar\hat{S}\g^*[-1]\ltimes\hat{S}_{\geq 1}\g^*[-1]\otimes\gl(\g^*)$, we obtain the following commutative diagram where, as above, the rows are $L_\infty$-extensions and vertical arrows are graded Lie algebra maps or strict $L_\infty$-maps:
\begin{equation}\label{eq:relation}
\xymatrix{
	  (\hat{S}\g[1-n])[n-1]\ar[r]\ar[d]^H& \huaL_{\LHM}(\g)\ar[d]\ar[r]& \Derbar\hat{S}\g^*[-1]\ar[d]^{H\circ D_n}	\\
  \Hom(\g^*[-1],\hat{S}\g[1-n])\ar[r]&\huaL_{\HRB}(\g,\g^*[n-2])\ar[r]&\huaL_{\HLR}(\g,\g^*[n-2])
}
\end{equation}
The above discussion can be summarized as follows.
\begin{theorem}
	Let $r$ be an $n$-shifted $r_\infty$-matrix in an $L_\infty$-algebra $(\g,m)$, i.e. the pair $(\g,m,r)$ is an $n$-shifted triangular $L_\infty$-bialgebra. Then $H(r)$ is a homotopy relative $\RB$ operator on $\g$ with respect to the coadjoint representation $\ad^*$ of $\g$ on $\g^*[n-2]$. The resulting correspondence \[(m,r)\mapsto(m,\ad^*,  H(r) )\] between triangular $L_\infty$-bialgebras and homotopy relative $\RB$ Lie algebras is realized as a strict $L_\infty$-map
	\[
	\huaL_{\LHM}(\g)\to \huaL_{\HRB}(\g,\g^*[n-2])
	\] between $L_\infty$-algebras governing the corresponding structures. \qed
\end{theorem}

\begin{rem}\label{rmk:relation}
  When we consider the ungraded case (i.e. when $\g$ is an ordinary Lie algebra), the strict $L_\infty$-map $\huaL_{\LHM}(\g)\to \huaL_{\HRB}(\g,\g^*[n-2])$ constructed above, recovers the map $\mathfrak i$ given in \cite[Proposition 4.19]{LTS}; note that $\mathfrak i$ was only constructed in op. cit. as a chain map rather than a strict $L_\infty$-map.
\end{rem}
\begin{rem}
	It is easy to see that the image of the map $\huaL_{\LHM}(\g)\to \huaL_{\HRB}(\g,\g^*[n-2])$ is contained inside an $L_\infty$-algebra whose underlying graded vector space is
	\[\Hom(\g^*[-1],\hat{S}\g[1-n])\oplus (H\circ D_n)(\Derbar\hat{S}\g^*[-1])\cong \Hom(\g^*[-1],\hat{S}\g[1-n])\oplus \Derbar\hat{S}\g^*[-1]\] and which governs homotopy RB Lie algebras on $\g$ with the fixed coadjoint representation. We omit the details.
\end{rem}

\noindent{\em Acknowledgements.} This research was partially supported by NSFC (11922110).  This work was completed in part while the first author was visiting Max Planck Institute for Mathematics in Bonn and he wishes to thank this institution for excellent working conditions.


\begin{thebibliography}{20}



\bibitem{Bai}
C. Bai, A unified algebraic approach to the classical Yang-Baxter equation. \emph{J. Phys. A } {\bf 40} (2007), 11073-11082.

\bibitem{Bandiera} R. Bandiera, Nonabelian higher derived brackets. {\em J. Pure Appl. Algebra} {\bf 219}  (2015),  3292-3313.

\bibitem{BSZ}	C. Bai, Y. Sheng and C. Zhu, Lie 2-bialgebras. {\em Comm. Math. Phys.} {\bf 320} (2013), 149-172.

\bibitem{BVor} D. Bashkirov and A. Voronov, $r_\infty$-Matrices, triangular $L_\infty$-bialgebras, and quantum$_{\infty}$ groups. Geometric Methods in Physics. XXXIII Workshop 2014, 39-47, Trends in Mathematics, Birkh\"auser Boston, Inc., Boston, MA, 2015.

\bibitem{Baxter} G. Baxter, An analytic problem whose solution follows from a simple algebraic identity. {\em Pacific J. Math.} {\bf 10} (1960), 731-742.

\bibitem{BLaz}
C. Braun and A. Lazarev, Unimodular homotopy algebras and Chern-Simons theory. {\em J. Pure  Appl. Algebra}  {\bf219} (2015), 5158-5194.

\bibitem{Bordemann}
M.  Bordemann,  An unabelian version of the Voronov higher bracket construction. {\em Georgian Math. J.} {\bf 22}  (2015),   189-214.

\bibitem{Cattaneo-Felder} A. Cattaneo and G. Felder,
Relative formality theorem and quantisation of coisotropic submanifolds.
{\em Adv. Math.} {\bf 208} (2007),  521-548.
\bibitem{CL'}
J. Chuang and A. Lazarev, {$L_\infty$-maps and twistings}. {\em Homology, Homotopy  Appl.} {\bf 13}  (2011), 175-195.

\bibitem{CL}	
J. Chuang and A. Lazarev, Combinatorics and formal geometry of the Maurer-Cartan equation. {\em Lett. Math. Phys.} {\bf 103}  (2012), 79-112.

\bibitem{CP}
V. Chari and A. Pressley, A Guide to Quantum Groups, Cambridge University Press, 1994.

\bibitem{CK}
A. Connes and D. Kreimer, { Renormalization in quantum field theory and the Riemann-Hilbert problem. I. The Hopf algebra structure of graphs and the main theorem.} {\em Comm. Math. Phys.} {\bf 210} (2000), 249-273.

\bibitem{Dolgushev-Rogers}
V. A. Dolgushev and C. L. Rogers, A version of the Goldman-Millson Theorem for filtered $L_\infty$-algebras. {\em J. Algebra}  {\bf 430} (2015), 260--302.

\bibitem{Etingof-Kazhdan}
P. Etingof, D.  Kazhdan,
Quantization of Lie bialgebras. I.
{\em Selecta Math. (N.S.)} {\bf 2} (1996),  1-41.

\bibitem{Fregier-Zambon-1}
Y. Fr\'egier and M. Zambon, Simultaneous deformations and Poisson geometry. \emph{ Compos. Math.} {\bf 151} (2015), 1763-1790.

 \bibitem{FZ} Y. Fr\'egier and M. Zambon, Simultaneous deformations of algebras and morphisms via derived brackets. {\em J. Pure Appl. Algebra} {\bf 219} (2015), 5344-5362.

	\bibitem{GLTS} A. Guan, A. Lazarev, Y. Sheng and R. Tang, {Review of deformation theory II: a homotopical approach}. \emph{Adv. Math. (China)} {\bf49}  (2020), 278-298.

\bibitem{Gub}
L. Guo,  An introduction to Rota-Baxter algebra. Surveys of Modern Mathematics, 4. International Press, Somerville, MA; Higher Education Press, Beijing, 2012. xii+226 pp.

\bibitem{Kon} M. Kontsevich, Deformation quantization of Poisson manifolds. \emph{Lett. Math. Phys.} {\bf 66} (2003), 157-216.

\bibitem{Kra} O. Kravchenko, Strongly Homotopy Lie Bialgebras and Lie Quasi-bialgebras. {\em Lett.  Math. Phys.} {\bf 81} (2007),  19-40.	

\bibitem{Kuper}  B. A. Kupershmidt, What a classical r-matrix really is. {\em J. Nonlinear Math. Phys.} {\bf 6} (1999), 448-488.

\bibitem{LS}
T. Lada and J. Stasheff, Introduction to sh Lie algebras for physicists. \emph{Internat. J. Theoret. Phys.} {\bf 32} (1993), 1087-1103.

\bibitem{LM}
T. Lada and M. Markl,  Strongly homotopy Lie algebras. \emph{Comm. Algebra} {\bf 23} (1995),  2147-2161.

\bibitem{Laz}
A. Lazarev, {Maurer-Cartan moduli and models for function spaces}. {\em Adv. Math.} {\bf 235}  (2013), 296-320.

\bibitem{LTS}
A. Lazarev, Y. Sheng and  R. Tang, \emph{Deformations and homotopy theory of relative Rota-Baxter Lie algebras}. The MPIM preprint series, 2020.

\bibitem{Laz'} A. Lazarev, {Models for classifying spaces and derived deformation theory}. {\em Proc. Lond. Math. Soc.} {\bf 109} (2014), 40-64.



\bibitem{Man} M. Manetti, Lie description of higher obstructions to deforming submanifolds. {\em  Ann. Sc. Norm. Super. Pisa Cl. Sci.}    {\bf6 } (2007),   631-659.



\bibitem{Me} R. Mehta and  M. Zambon, {L-infinity algebra actions}. {\em Diff. Geom.  Appl.} {\bf 30} (2012), 576-587.

\bibitem{Mer} S. Merkulov, PROP Profile of Poisson Geometry. {\em Comm. Math. Phys.} {\bf 262} (2006), 117-135.

\bibitem{Mer2} S. Merkulov, Frobenius$_{\infty}$ invariants of homotopy Gerstenhaber algebras. I.  \emph{Duke Math. J.} {\bf 105}  (2000),  411-461.

\bibitem{Merkulov} S. Merkulov, {Formality theorem for quantizations of Lie bialgebras}. {\em Lett. Math. Phys.} {\bf 106}  (2016),   169-195.

\bibitem{Resh}N. Reshetikhin, Quantization of Lie bialgebras. {\em Internat. Math. Res. Notices}
{\bf 7} (1992), 143-151.

\bibitem{Rota} G.-C. Rota, Baxter algebras and combinatorial identities I, II. {\em Bull. Amer. Math. Soc.} {\bf75} (1969), 325-334.

\bibitem{STS}  M. A. Semyonov-Tian-Shansky, What is a classical R-matrix? {\em Funct. Anal. Appl.} {\bf 17} (1983), 259-272.




\bibitem{Vor} T. Voronov, {Higher derived brackets and homotopy algebras.} \emph{J. Pure Appl. Algebra} {\bf 202} (2005), 133-153.	

	\end{thebibliography}
\end{document}